\documentclass[10pt]{article}
\usepackage{amsthm}
\usepackage{amsxtra}
\usepackage{amssymb}
\usepackage{thmtools}
\usepackage{enumitem}
\usepackage[colorlinks=true,linkcolor=blue]{hyperref}
\usepackage{mathrsfs}
\usepackage[retainorgcmds]{IEEEtrantools}
\usepackage{undertilde}
\usepackage{url}

\declaretheoremstyle[bodyfont=\sl]{slanted}

\swapnumbers
\declaretheorem[name=Definition,style=definition,qed=$\dashv$,
numberwithin=section]{dfn}
\declaretheorem[name=Definition,style=definition,numbered=no,qed=$\dashv$]{dfn*}
\declaretheorem[name=Definition,style=definition,numbered=no]{dfnnoqed*}

\declaretheorem[name=Theorem,style=slanted,sibling=dfn]{tm}
\declaretheorem[name=Theorem,style=slanted,numbered=no]{tm*}
\declaretheorem[name=Lemma,style=slanted,sibling=dfn]{lem}

\declaretheorem[name=Corollary,style=slanted,numbered=no]{cor*}
\declaretheorem[name=Remark,style=definition,sibling=dfn]{rem}

\swapnumbers
\declaretheoremstyle[headfont=\scshape]{claimstyle}
\declaretheorem[name=Claim,style=claimstyle]{clm}

\declaretheorem[name=Claim,style=claimstyle,numbered=no]{clm*}

\declaretheorem[name=Subclaim,style=claimstyle,numbered=no]{sclm*}

\declaretheorem[name=Subsubclaim,style=claimstyle,numberwithin=sclmtwo]{ssclmtwo
}
\declaretheorem[name=Subsubclaim,style=claimstyle,numberwithin=sclmthree]{
ssclmthree}
\declaretheorem[name=Subsubclaim,style=claimstyle,numberwithin=sclmfour]{
ssclmfour}
\declaretheorem[name=Subsubclaim,style=claimstyle,numberwithin=sclmfive]{
ssclmfive}
\declaretheorem[name=Subsubclaim,style=claimstyle,numberwithin=sclmsix]{ssclmsix
}
\declaretheorem[name=Subsubclaim,style=claimstyle,numberwithin=sclmseven]{
ssclmseven}
\declaretheorem[name=Subsubclaim,style=claimstyle,numberwithin=sclmeight]{
ssclmeight}
\declaretheorem[name=Subsubclaim,style=claimstyle,numberwithin=sclmnine]{
ssclmnine}
\declaretheorem[name=Subsubclaim,style=claimstyle,numberwithin=sclmten]{ssclmten
}
\declaretheorem[name=Subsubclaim,style=claimstyle,numbered=no]{ssclm*}

\declaretheoremstyle[headfont=\scshape]{casestyle}

\declaretheorem[name=Case,style=casestyle]{case}
\declaretheorem[name=Case,style=casestyle]{casetwo}
\declaretheorem[name=Case,style=casestyle]{casethree}

\declaretheorem[name=Subcase,style=casestyle,numberwithin=case]{scase}

\declaretheorem[name=Subcase,style=casestyle,numberwithin=casethree]{scasethree}

\declaretheorem[name=Subsubcase,style=casestyle,numberwithin=scasethree]
{sscasethree}

\newcommand{\conv}{\mathrm{conv}}

\newcommand{\compmode}{1}%0 is standard, 1 is ...
\newcommand{\compopt}[2]{\ifthenelse{\equal{\compmode}{0}}{#1}{#2}}

\newcommand{\lgcd}{\mathrm{lgcd}}

\newcommand{\direct}{\mathrm{direct}}

\newcommand{\CC}{\mathbb C}

\newcommand{\sub}{\subseteq}

\newcommand{\inter}{\cap}
\renewcommand{\int}{\inter}

\newcommand{\om}{\omega}
\newcommand{\pow}{\mathcal{P}}
\newcommand{\OR}{\mathrm{OR}}

\newcommand{\cut}{\backslash}

\newcommand{\Tt}{\mathcal{T}}
\newcommand{\Ss}{\mathcal{S}}
\newcommand{\Uu}{\mathcal{U}}
\newcommand{\Vv}{\mathcal{V}}
\newcommand{\Ww}{\mathcal{W}}

\newcommand{\rg}{\mathrm{rg}}

\newcommand{\ins}{\trianglelefteq}
\newcommand{\nins}{\ntrianglelefteq}
\newcommand{\pins}{\triangleleft}

\newcommand{\crit}{\mathrm{cr}}

\newcommand{\rest}{\!\upharpoonright\!}
\newcommand{\com}{\circ}

\newcommand{\lh}{\mathrm{lh}}
\newcommand{\Ult}{\mathrm{Ult}}

\newcommand{\sats}{\models}

\newcommand{\J}{\mathcal{J}}

\newcommand{\ZFC}{\mathsf{ZFC}}

\newcommand{\es}{\mathbb{E}}

\newcommand{\core}{\mathfrak{C}}

\newcommand{\pred}{\mathrm{pred}}

\newcommand{\id}{\mathrm{id}}

\newcommand{\sq}{\mathrm{sq}}

\newcommand{\conc}{\ \widehat{\ }\ }

\newcommand{\rSigma}{\mathrm{r}\Sigma}

\DeclareMathOperator{\card}{card}
\DeclareMathOperator{\cof}{cof}

\newcommand{\cHull}{\mathrm{cHull}}

\newcommand{\unrvl}{\mathrm{unrvl}}

\renewcommand{\cut}{\backslash}

\renewcommand{\Ss}{\mathcal{S}}

\newcommand{\pvec}{\vec{p}}

\renewcommand{\hbar}{\bar{h}}

\newcommand{\udash}{\mathrm{u}\text{-}}
\newcommand{\udeg}{{\udash\deg}}

\newcommand{\uu}{\mathrm{u}}

\renewcommand{\pm}{\mathrm{pm}}
\newcommand{\exit}{\mathrm{ex}}

\newcommand{\Gg}{\mathcal{G}}

\newcommand{\successor}{\mathrm{succ}}

\newcommand{\passive}{\mathrm{pv}}
\newcommand{\dropset}{\mathscr{D}}

\newcommand{\myurl}{https://sites.google.com/site/schlutzenberg}

\begin{document}

\title{Background construction for $\lambda$-indexed mice}

\author{Farmer Schlutzenberg
\footnote{Partially funded by the Deutsche Forschungsgemeinschaft (DFG, 
German Research Foundation) under Germany's Excellence Strategy EXC
2044-390685587, Mathematics M\"unster: Dynamics--Geometry--Structure.
}\ \footnote{
farmer.schlutzenberg@gmail.com,
\url{\myurl}}\\
WWU M\"unster}

\maketitle

\newcommand{\ext}{\mathrm{ext}}

\begin{abstract}
Let $M$ be a $\lambda$-indexed (that is, Jensen indexed) premouse. We prove that 
$M$ is iterable with respect to standard $\lambda$-iteration rules iff $M$ is 
iterable with respect to a natural version of Mitchell-Steel iteration rules. 
Using this equivalence, we describe a background construction for 
$\lambda$-indexed mice, analogous to traditional background constructions for 
Mitchell-Steel indexed mice, and which absorbs Woodin cardinals from the 
background universe.

We also prove some facts regarding the correspondence between standard iteration 
trees and u-iteration trees on premice with Mitchell-Steel indexing; these facts 
 were stated and used without proof in \cite{iter_for_stacks}.\end{abstract}

\section{Introduction}

There are two standard forms of fine structural premice
commonly used in the inner model theoretic literature:
those with \emph{Mitchell-Steel indexing (MS-indexing, MS-premice)},
and those with \emph{$\lambda$-indexing} or \emph{Jensen indexing (J-indexing, 
J-premice)}.
Let $M$ be an active premouse of either kind,
$F=F^M$ the active extender of $M$,
$\kappa=\crit(F)$ its critical point, $\nu=\nu(F)$
the strict sup of generators, and $\lambda=\lambda(F)=i^M_F(\kappa)$
where $i^M_F:M\to U=\Ult(M,F)$ is the ultrapower map.
We have $\nu\leq\lambda$.
If $M$ is an MS-premouse then
$\OR^M=(\nu^+)^U$, whereas if $M$ is a J-premouse
then $\OR^M=(\lambda^+)^U$.  Also, iteration
trees on these premice are usually formed according to different rules:
let $\Tt$ be a normal tree on $M$. If $M$ is MS-indexed,
this conventionally mean that $\pred^\Tt(\alpha+1)$
is the least $\beta$ such that $\crit(E^\Tt_\alpha)<\nu(E^\Tt_\beta)$,
whereas if $M$ is J-indexed, it conventionally means that
$\pred^\Tt(\alpha+1)$ is the least $\beta$ such that
$\crit(E^\Tt_\alpha)<\lambda(E^\Tt_\beta)$. 
Let us call these two rules \emph{MS-rules}
and \emph{J-rules} respectively. 
If $M$ is J-indexed, it also makes sense to form trees
with MS-rules, although this is not normally done.
(Actually, we will deal with another variant of MS-rules,
as it is what  arises naturally with background construction.)

J-premice are in certain ways easier to deal with.
However, one traditional advantage of MS-premice
is that (when we use the traditional iteration rules respectively)
backgrounded $L[\es]$-constructions for MS-premice absorb all Woodin 
cardinals from the background universe $R$ (assuming $R$ is
sufficient iterability in some larger universe), whereas constructions
for J-premice require stronger large cardinals in $R$ in order
to reach Woodin cardinals.
This is because in order to get iterability with respect
to J-rules, the traditional methods have required
that when an active J-premouse $M$ has $F^M$ induced
by 
background extender $F^*$,  $F^*$ should be at least $\lambda(F^M)$-strong;
this ensures in particular that the usual procedure
for lifting trees on $M$ to trees on $R$ makes sense.

Also of relevance is
Fuchs' translation between
MS-mice and J-mice in \cite{fuchs_translation_1},
and respective iteration strategies for them in \cite{fuchs_translation_2}.
His translation of iteration strategies works, however,
 only with strategies for MS-rules,
and thus does not give the translation between
the standard forms of iterability of the respective forms of mice.

In this note 
we show
that for $k$-sound J-premice,
for example, $(k,\om_1+1)$-iterability
with respect to  J-rules is equivalent 
to that with respect to a natural form of MS-rules
(\emph{lifting-MS-rules}), and in fact,
there is a one-to-one correspondence between
witnessing strategies.  In fact, even without iterability,
there is in  fact a 
direct translation
between iteration trees of the two kinds, given
that they satisfy some slight extra properties (which
follow immediately if they are according to $(k,\om_1+1)$-strategies).
Moreover, the correspondence extends easily to $k$-maximal stacks.
A straightforward variant also establishes
equivalence with strategies for MS-rules as described earlier.

Using this, we then define a form of backgrounded
$L[\es]$-construction for J-premice, and show that it
absorbs Woodin cardinals from the background universe,
like that for MS-premice. The method for the
background construction is similar to that in \cite[\S5]{premouse_inheriting}.

Using a very similar kind of procedure
for translation between different iteration rules,
but for MS-premice,
we also prove some facts used in \cite{iter_for_stacks}
(which were stated there without proof).\footnote{The basic
idea for the tree conversion process for J-trees
was observed by the author
in 2014 or slightly earlier. After mentioning the basic method
 to Steel and Schindler during a lunch at the \emph{A2 am See}
in mid 2015, it seemed apparent that 
it had escaped notice 
prior to that point, and in particular was
not covered in  Fuchs \cite{fuchs_translation_2}. So it seemed
it should be written up.
The variant for MS-trees was observed in 2015,
and the background construction for J-mice
during 2018.}

\subsection{Notation}

See \cite{iter_for_stacks} for a guide to notation,
and in particular \cite[\S2]{iter_for_stacks}
for background on $\udash$fine structure (which
is only relevant in \S\ref{sec:MS_translation}).

A premouse is an \emph{MS-premouse} or is \emph{MS-indexed}
iff it is a premouse with Mitchell-Steel indexing, but allowing 
superstrong extenders as discussed in \cite{iter_for_stacks}.
A premouse is a \emph{J-premouse} or is \emph{J-indexed}
iff it has Jensen indexing (also known as $\lambda$-indexing). A seg-pm $M$ is 
\emph{internally MS-indexed}
iff $M^\passive$ is MS-indexed, and is \emph{internally J-indexed}
iff $M^\passive$ is J-indexed.
We write $\iota(M)=\max(\nu(F^M),\lgcd(M))$
 for an active seg-pm $M$.

For normal iteration trees $\Tt$ (where \emph{normality}
depends on context), we write $\widetilde{\nu}^\Tt_\beta$
for the exchange ordinal associated to $M^\Tt_\beta$;
that is, $\pred^\Tt(\alpha+1)$ is the least $\beta$
such that $\crit(E^\Tt_\alpha)<\widetilde{\nu}^\Tt_\beta$.

We will often deal with padded fine structural trees $\Tt$.
When $E^\Tt_\alpha=\emptyset$, we always
set $\pred^\Tt(\alpha+1)=\alpha$,
 $M^\Tt_{\alpha+1}=M^\Tt_\alpha$,
 $\deg^\Tt_{\alpha+1}=\deg^\Tt_\alpha$,
 and $i^\Tt_{\alpha,\alpha+1}=\id$.

 \section{Tree conversion for J-mice}\label{sec:J_conversion}

 \begin{dfn}
Let $M$ be an $m$-sound J-premouse.
An \emph{J-$m$-maximal} tree $\Tt$ on $M$
is one formed with the usual exchange ordinals
for iterating such mice ($\widetilde{\nu}^\Tt_\alpha=\lambda(E^\Tt_\alpha)$),
and otherwise, the usual conditions for $m$-maximality.
A \emph{J-$(m,\Omega+1)$-iteration strategy} for $M$
is an $(m,\Omega+1)$-iteration strategy for $M$ with respect to such trees.
A \emph{weakly-MS-$m$-maximal} tree $\Tt$ on $M$
is likewise, but with exchange ordinals 
$\widetilde{\nu}^\Tt_\alpha=\nu(E^\Tt_\alpha)$.
An \emph{weakly-MS-$(m,\Omega+1)$-iteration strategy}
for $M$ is one for such trees.
\end{dfn}

Let us start with describing the main point of the tree conversion process to 
be presented. Let $\Uu$ be a weakly-MS-$m$-maximal
tree on an $m$-sound J-premouse $M$.
Let $\nu=\nu(E^\Uu_0)$ and $\lambda=\lambda(E^\Uu_0)$
and $\kappa=\crit(E^\Uu_1)$, and suppose
that $\nu\leq\kappa<\lambda$ and $2\notin\dropset^\Uu$.
Then $\pred^\Uu(2)=1$,
and note that in fact $2\notin\dropset_{\deg}^\Uu$,
since letting $k=\deg^\Uu(1)$, we have $
\lambda<\rho_k(M^\Uu_1)$.
Now suppose  we  form a J-$m$-maximal
tree $\Tt$ with $E^\Tt_0=E^\Uu_0$ and $E^\Tt_1=E^\Uu_1$.
Then we would instead set $\pred^\Tt(2)=0$,
and noting that $\rho_1(\exit^\Tt_0)\leq\nu$
(where $\exit^\Tt_\alpha=M^\Tt_\alpha|\lh(E^\Tt_\alpha)$),
we get $M^{*\Tt}_2=\exit^\Tt_0$ and $\deg^\Tt_2=0$,
so $M^\Tt_2=\Ult_0(\exit^\Tt_0,E^\Tt_1)$.
But $(\exit^\Tt_0)^\passive$ is a cardinal segment
of $M^\Uu_1$ and $\lh(E^\Tt_0)\leq\rho_k(M^\Uu_1)$,
and it follows that $\Ult_0(\exit^\Tt_0,E^\Tt_1)$
is a cardinal segment of $M^\Uu_2$,
and moreover,
\[ i^{*\Tt}_2\sub i^{*\Uu}_2:M^\Uu_1\to M^\Uu_2.\]
Also, a standard computation shows that $F(M^\Tt_2)$
is equivalent to the two-step iteration given by 
$(E^\Tt_0,E^\Tt_1)=(E^\Uu_0,E^\Uu_1)$. (That is, 
forming a degree $k$ ultrapower  $U$ with $F(M^\Tt_2)$ gives the same
result as first forming an intermediate degree $k$ ultrapower $U'$ with 
$E^\Tt_0$,
then forming $U$ as the degree $k$ ultrapower of $U'$ with $E^\Tt_1$,
including agreement of the overall ultrapower maps.)
So let $\Tt'=\Tt\conc F(M^\Tt_2)$ (formed as a J-$m$-maximal tree).
Then note that $M^{*\Tt'}_3=M^{*\Uu}_1$
(since $F(M^\Tt_2)$ and $E^\Tt_0=E^\Uu_0$ have the same critical
point and measure space) and $\deg^{\Tt'}_3=\deg^\Uu_1=\deg^\Uu_2=k$,
and so $M^{\Tt'}_3=M^\Uu_2$ and $i^{*\Tt'}_3=i^\Uu_{12}\com i^{*\Uu}_1$.

The correspondence described above is the key component
of the conversion process in general. In more generality,
the nesting of extenders (like how $E^\Tt_1$ sits inside $F(M^\Tt_2)$ above)
can be arbitrarily finitely deep, 
and we need to keep track of transfinite concatenations of such things,
and the resulting modifications in tree order between
an (appropriate) J-$m$-maximal tree $\Tt$ and 
the corresponding $\Uu$.
The entire conversion process
is explicit and locally computed (in both directions, $\Tt$ to $\Uu$
and $\Uu$ to $\Tt$).

We will actually not discuss the process explicitly
for weakly-MS-$m$-maximal trees, but instead
for another form, lifting-MS-$m$-maximal trees, which arise naturally in the 
background construction.
But a very similar process can also be used for weakly-MS-$m$-maximal trees.
For the conversion between J-trees $\Tt$ and lifting-MS-trees $\Uu$,
the disagreements in tree order and models etc
also arise in a slightly different fashion
(because  $\Uu$ uses extra extenders,
which appear in  the dropdown sequence of extenders
used in $\Tt$). But the basic point of the correspondence
is still that described above. The rest of the work is mostly
a matter of bookkeeping, but somehow when one tries to 
write things down in detail, that bookkeeping
seems to run to some length.

\begin{dfn} Let $M$ be a J-pm and $N\ins M$ with $N$ active.

 The \emph{mod-dropdown} of $(M,N)$ is the sequence $\left<N_i\right>_{i\leq 
k}$,
 with $k<\om$ as large as possible, such that $N_0=N$
 and $N_{i+1}$ is the least $N'$ such that $N_i\pins N'\ins M$
 and either (i) $N'\pins M$ and $\rho_\om^{N'}<\rho_\om^{N_i}$,
 or (ii) $N'=M$.
 
 Let $\left<N_i\right>_{i\leq k}$ be the mod-dropdown of $(M,N)$.
 The \emph{reverse-e-dropdown} (\emph{rev-e-dropdown}) of $(M,F^N)$ is the 
sequence 
$\left<E_i\right>_{i\leq n}$
 enumerating all extenders of the form $F^{N_i}$ where $i\leq k$
 and $N_i$ is active and if $i>0$ then $\nu(F^{N_i})<\rho_\om^{N_{i-1}}$,
 in order of decreasing index $\OR^{N_i}$.
\end{dfn}
\begin{lem}\label{lem:rev-e-dd_seg}
 Let $\left<E_i\right>_{i\leq n}$ be the rev-e-dropdown of $(M,E)$ (so 
$E=E_n$).
 Then:
 \begin{enumerate}
  \item\label{item:nu(E_i)<lambda(E_i)} Let $i<n$.
Let $N'$ be the largest element of the mod-dropdown
of $(M,E_n)$ such that $N'\pins M|\lh(E_i)$ (so $M|\lh(E_{i+1})\ins N'$).
Then
\[ \nu(E_i)<\lambda(E_i)\leq\rho_\om^{N'}\leq\nu(E_{i+1}).\]
In fact, writing $\nu_k=\nu(E_k)$ and $\lambda_k=\lambda(E_k)$
and $\lh_k=\lh(E_k)$,
either:
\begin{enumerate}[label=--]
\item 
$\nu_i<\lambda_i=\rho_\om^{N'}\leq\rho_\om^{M|\lh_{i+1}}\leq
\nu_{i+1}\leq\lambda_{i+1}<\lh_{i+1}\leq\OR^{N'}<\lh_i$, or
\item 
$\nu_i<\rho_\om^{N'}\leq\rho_\om^{M|\lh_{i+1}}\leq
\nu_{i+1}\leq\lambda_{i+1}<\lh_{i+1}\leq\OR^{N'}<\lambda_i<\lh_i$.
\end{enumerate}
\item For each $i\leq n$,
the reverse-e-dropdown of $(M,E_i)$
 is $\left<E_k\right>_{k\leq i}$.
\item\label{item:rev-e-dd_of_Ult} Let $i<n$ and 
$U=\Ult_k(N,E_i)$, where $N$ is a $k$-sound J-premouse
and $E_i$ an $N$-extender with $\crit(E_i)<\rho_k^N$,
and if $N$ is active
then $\crit(E_i)<\lambda(F^N)$. Suppose $U$ is wellfounded.
Then the reverse-e-dropdown of $(U,E_n)$ is:
\begin{enumerate}
 \item $\left<F^U\right>\conc\left<E_k\right>_{i+1\leq k\leq n}$,
 if $N$ is 
active
 and $\nu(F^N)\leq\crit(E_i)$,
 \item $\left<E_{k}\right>_{i+1\leq k\leq n}$, otherwise.
\end{enumerate}
\end{enumerate}
\end{lem}
\begin{proof}
This is basically a straightforward
consequence of the definitions, and left to the reader.
In part \ref{item:rev-e-dd_of_Ult},
if $N$ is active with $\nu(F^N)\leq\crit(E_i)$,
use coherence
and that $\lh(E_i)$ is a $U$-cardinal to get
appropriate agreement of the mod-dropdowns of $M$ and $U$,
and use that $\nu(F^U)=\nu(E_i)$, which is by a
standard calculation
(see for example \cite[Lemma 2.11***]{extmax}).
If instead $N$ is active
with $\crit(E_i)<\nu(F^N)$,
and $\gamma$ is a generator of $F^N$,
then $j(\gamma)$ is a generator of $F^U$,
where $j:N\to U$ is the ultrapower map,
and hence $\nu(F^U)\geq j``\nu(F^N)>\lambda(E_i)\geq\rho_\om^{N'}$,
where $N'$ is as in part \ref{item:nu(E_i)<lambda(E_i)}.
Note $N'$ is in the mod-dropdown of $(U,E_n)$, 
so $F^U$ is not in the 
rev-e-dropdown
in this case.
\end{proof}

\begin{dfn}Let $M$ be an $m$-sound J-pm and $\Tt$ a putative iteration tree on 
$M$.
We say that $\Tt$ is \emph{lifting-MS-$m$-maximal} iff
there is an ordinal $\lambda>0$ and sequences 
$\left<\eta_\alpha\right>_{\alpha<\lambda}$,
$\left<n_\alpha\right>_{\alpha+1<\lambda}$ such that:
\begin{enumerate}
 \item $\eta_0=0$ and $\left<\eta_\alpha\right>_{\alpha<\lambda}$ is continuous.
 \item $n_\alpha<\om$ and $\eta_{\alpha+1}=\eta_\alpha+n_\alpha+1$.
 \item If $\lambda=\alpha+1$ then $n_\alpha=0$.
 \item $\Tt$ has length $\eta=\sup_{\alpha<\lambda}(\eta_\alpha+1)$.
 \item If $\gamma+1<\lh(\Tt)$ then $E^\Tt_\gamma\in\es_+(M^\Tt_\gamma)$.
 \item\label{item:rev_e-drop} If $\alpha+1<\lambda$ then 
$\left<E^\Tt_{\eta_\alpha},E^\Tt_{\eta_\alpha+1},\ldots,E^\Tt_{
\alpha_\alpha+n_\alpha}\right>$
 is the rev-e-dropdown of 
$(M^\Tt_{\eta_\alpha},E^\Tt_{\eta_\alpha+n_\alpha})$.
 \item If $\alpha+1<\beta+1<\lambda$ then 
$\lh(E^\Tt_{\eta_\alpha+n_\alpha})<\lh(E^\Tt_{\eta_\beta+n_\beta})$.
 \item If $\gamma+1<\lh(\Tt)$ then $\pred^\Tt(\gamma+1)$ is the least $\delta$ 
such that $\crit(E^\Tt_\gamma)<\nu(E^\Tt_\delta)$,
 and $M^{*\Tt}_{\gamma+1}\ins M^\Tt_\delta$ and $\deg^\Tt(\gamma+1)$ are 
determined as usual for $m$-maximality.\qedhere
\end{enumerate}
\end{dfn}

\begin{rem} Note that given $\Tt\rest(\eta_\alpha+1)$,
we can choose any $E\in\es_+(M^\Tt_{\eta_\alpha})$
such that $\lh(E^\Tt_{\eta_\beta+n_\beta})<\lh(E)$ for all $\beta<\alpha$,
and then extend $\Tt$ by satisfying condition \ref{item:rev_e-drop}
with $E^\Tt_{\eta_\alpha+n_\alpha}=E$, given that the resulting ultrapowers are 
wellfounded.
In particular, $E^\Tt_{\eta_\alpha+i}\in\es_+(M^\Tt_{\eta_\alpha+i})$
for each $i\leq n_\alpha$, by coherence.

Also note that by \ref{lem:rev-e-dd_seg} and coherence,
for all $j,k$ with  $j\leq k\leq n_\alpha$,
\[ 
\left<E^\Tt_{\eta_\alpha+j},E^\Tt_{\eta_\alpha+j+1},\ldots,E^\Tt_{\eta_\alpha+k}
\right> \]
is a \emph{tail segment} of the rev-e-dropdown of 
$(M^\Tt_{\eta_\alpha+j},E^\Tt_{\eta_\alpha+k})$.
It follows from Lemma \ref{lem:nu_inc} below that it is actually the 
\emph{full} 
rev-e-dropdown.
\end{rem}

\begin{lem}\label{lem:nu_inc}Let $M$ be an $m$-sound J-premouse and $\Tt$ be 
lifting-MS-$m$-maximal on $M$.
 Then:
 \begin{enumerate}
  \item\label{item:nu_increasing} $\nu(E^\Tt_\gamma)<\nu(E^\Tt_\delta)$ for all 
$\delta+1<\gamma+1<\lh(\Tt)$.
  \item\label{item:active_nu} Let $\gamma<\lh(\Tt)$ be such that 
$F=F^{M^\Tt_{\gamma}}\neq\emptyset$.
  Then $\lambda(E^\Tt_{\eta_\alpha+n_\alpha})<\nu(F)$ for all $\alpha$
  such that $\eta_\alpha+n_\alpha<\gamma$,
  and if $\gamma=\delta+1$ then $\lambda(E^\Tt_\delta)<\nu(F)$.
  \item\label{item:crit<nu} For all $\gamma+1<\lh(\Tt)$, if 
$M^{*\Tt}_{\gamma+1}$ is active
  then $\crit(E^\Tt_\gamma)<\nu(F(M^{*\Tt}_{\gamma+1}))$.
 \end{enumerate}

\end{lem}
\begin{proof}
We prove that the lemma holds for $\Tt\rest\xi$ by induction on $\xi$.
Suppose that it holds for $\Tt\rest\xi$; we prove it for $\Tt\rest\xi+1$.

If $\xi$ is a limit then we just have to verify part \ref{item:active_nu}
holds when $\gamma=\xi$, but this is easy.

So suppose $\xi=\gamma+1$.

Part \ref{item:nu_increasing}: Let $\delta<\gamma$.
If there is $\alpha$ such that 
$\delta,\gamma\in[\eta_\alpha,\eta_\alpha+n_\alpha]$,
then $\nu(E^\Tt_\delta)<\nu(E^\Tt_\gamma)$ by definition of the 
reverse-e-dropdown.
So we may assume that $\delta=\eta_\alpha+n_\alpha$ for some $\alpha$.
Then $\lh(E^\Tt_\delta)$ is a cardinal of $M^\Tt_\gamma$
and $\lh(E^\Tt_\delta)<\lh(E^\Tt_\gamma)$. So if $E^\Tt_\gamma\neq 
F(M^\Tt_\gamma)$
then easily $\nu(E^\Tt_\delta)<\lh(E^\Tt_\delta)\leq\nu(E^\Tt_\gamma)$, so we 
are done.
And if $E^\Tt_\gamma=F(M^\Tt_\gamma)$
then by induction part \ref{item:active_nu},
we have $\nu(E^\Tt_\delta)\leq\lambda(E^\Tt_\delta)<\nu(E^\Tt_\gamma)$, as 
required.

Part \ref{item:crit<nu}:
 Suppose not. Let $\delta=\pred^\Tt(\gamma+1)$ and $\kappa=\crit(E^\Tt_\gamma)$.
 So
 \[ \nu(F(M^{*\Tt}_{\gamma+1}))\leq\kappa<\nu(E^\Tt_\delta), \]
 so $\exit^\Tt_\delta\pins M^{*\Tt}_{\gamma+1}$.
 There is no $N$ such that $\exit^\Tt_\delta\ins N\pins M^{*\Tt}_{\gamma+1}$
 and $\rho_\om^N\leq\kappa$. But then it follows that $F(M^{*\Tt}_{\gamma+1})$
 is in the reverse-e-dropdown of $(M^\Tt_\delta,\exit^\Tt_\delta)$.

 Now fix $\alpha$ such that $\delta\in[\eta_\alpha,\eta_\alpha+n_\alpha]$.
 If $\delta=\eta_\alpha$, then by \ref{lem:rev-e-dd_seg}, we should have used
 $F(M^{*\Tt}_{\alpha+1})$ in $\Tt$ before $E^\Tt_\delta$,
 contradiction. So $\delta=\eta_\alpha+i+1$ for some $i+1\leq n_\alpha$.
 Let $F=E^\Tt_{\eta_\alpha+i}$.
 Then $\lambda(F)$ and $\lh(F)$ are cardinals of $M^\Tt_\delta$,
 and we can't have $\lh(F)<\OR(M^{*\Tt}_{\gamma+1})$
 by coherence.
 But since
 \[ \rho_1(M^{*\Tt}_{\gamma+1})\leq\nu(M^{*\Tt}_{\gamma+1})<\lh(F),\]
 it follows that $M^{*\Tt}_{\gamma+1}=M^\Tt_\delta$.
 By induction, $\lambda(F)<\nu(F(M^\Tt_\delta))$.
 So
 \[ \lambda(F)<\nu(F(M^{*\Tt}_{\gamma+1}))\leq\kappa<\lh(E^\Tt_\delta)<\lh(F). 
\]
 But $\lh(F)=(\lambda(F)^+)^{M^\Tt_\delta}$, so there is some $N$ such that
 \[ \exit^\Tt_\delta\ins N\pins M^\Tt_\delta|\lh(F)\]
 and $\rho_\om^N\leq\kappa$. So $\OR(M^{*\Tt}_{\gamma+1})<\lh(F)$, 
contradiction.

 Part \ref{item:active_nu} follows from part \ref{item:crit<nu} as usual.
\end{proof}

\begin{dfn}
 Let $M$ be an $m$-sound J-pm.
 A \emph{lifting-MS-}$(m,\om_1+1)$-strategy for $M$
 is a winning strategy for player II in the iteration game producing (putative) 
lifting-MS-$m$-maximal trees,
 in which, given $\Tt\rest(\eta_\alpha+1)$ (notation as above),
 player I chooses $E\in\es_+(M^\Tt_{\eta_\alpha})$
 with $\lh(E^\Tt_{\eta_\beta+n_\beta})<\lh(E)$ for all $\beta<\alpha$,
 then $\Tt$ is extended with the reverse-e-dropdown
 of $(M^\Tt_\alpha,E)$, producing $\Tt\rest(\eta_{\alpha+1}+1)$;
 and player II chooses cofinal branches at limit stages.
 Player II must ensure every model produced is wellfounded, in order to win.
 
 We say that $M$ is \emph{lifting-MS-$(m,\om_1+1)$-iterable}
 iff there is a such an iteration strategy for $M$.
\end{dfn}

\begin{tm}\label{tm:lifting-iter_equiv}
Let $\Omega>\om$ be regular.
 Let $M$ be an $m$-sound J-pm.
 Then the following are equivalent:
 \begin{enumerate}[label=--]
  \item $M$ is J-$(m,\Omega+1)$-iterable,
  \item $M$ is weakly-MS-$(m,\Omega+1)$-iterable,
  \item $M$ is lifting-MS-$(m,\Omega+1)$-iterable.
 \end{enumerate}
 Moreover, there is a one-to-one correspondence
 between strategies of one kind and those of another,
 and given such a pair $(\Lambda,\Sigma)$
 of strategies, there is a direct
 (and local) translation between trees $\Tt$
 via $\Lambda$ and corresponding
 trees $\Uu$ via $\Sigma$.
\end{tm}
The precise meaning of statements at the end is clarified
in what follows.
\begin{dfn}
 Let $\Tt$ be a J-$m$-maximal tree
 on an $m$-sound J-premouse.
 Let $\alpha<\lh(\Tt)$. We say $\alpha$ is \emph{$\Tt$-special}
 iff $M^\Tt_\alpha$ is active and there is $\beta+1\leq^\Tt\alpha$ such that 
 $E^\Tt_\beta\neq\emptyset$ and
$(\beta+1,\alpha]_\Tt\inter\dropset^\Tt=\emptyset$
and $\nu(F(M^{*\Tt}_{\beta+1}))\leq\crit(E^\Tt_\beta)$.
We say $\alpha$ is \emph{$\Tt$-very-special} (\emph{$\Tt$-vs})
iff $\Tt$-special and $E^\Tt_\alpha=F(M^\Tt_\alpha)$.
If $\alpha$ is $\Tt$-special, and $\beta$ is the least witness,
let $\core(F(M^\Tt_\alpha))$ denote
$F(M^{*\Tt}_{\beta+1})$.

Let $E\in\es_+(M^\Tt_\alpha)$. Let
$\left<E_i\right>_{i\leq n}$ be the rev-e-dropdown
of $(M^\Tt_\alpha,E)$. The \emph{$\Tt$-despecialized-reverse-e-dropdown}
(\emph{$\Tt$-despec-rev-e-dropdown})
of $(M^\Tt_\alpha,E)$ is $\left<E_i\right>_{1\leq i\leq n}$,
if $\alpha$ is $\Tt$-special, and is $E_i=F(M^\Tt_\alpha)$,
and is $\left<E_i\right>_{i\leq n}$ otherwise.\footnote{So
when $\Tt$-despecializing, we just remove $E_0$ from 
the rev-e-dropdown
if $E_0=F(M^\Tt_\alpha)$ and $\alpha$ is $\Tt$-special.}
\end{dfn}

Note  $\alpha<^{\ext,\Tt}\beta\implies \alpha<\beta$.
So $<^{\ext,\Tt}$ is wellfounded.

\begin{lem}\label{lem:core_appears_J}
 Let $\Tt$ be a J-$m$-maximal tree
 on an $m$-sound J-premouse.
Let $\alpha$ be $\Tt$-special, $\beta$ the least
witness and $\gamma=\pred^\Tt(\beta+1)$.
Let $F=\core(F(M^\Tt_\alpha))=F(M^{*\Tt}_{\beta+1})$.
Then:
\begin{enumerate}
\item\label{item:F_in_rev-e-dd} $F$
is in the reverse-e-dropdown of
$(M^\Tt_\gamma,E^\Tt_\gamma)$, which
is identical with the $\Tt$-despecialized-reverse-e-dropdown
of $(M^\Tt_\gamma,E^\Tt_\gamma)$,
and $\gamma$ is the 
unique such ordinal.
 \item\label{item:nu(F)_leq_kappa_<_lambda(F)} 
$\nu(F)\leq\crit(E^\Tt_\beta)<\min(\lambda(E^\Tt_\gamma),\lambda(F))$,
 \item $\deg^\Tt_{\beta+1}=\deg^\Tt_\alpha=0$,
\item\label{item:ext_assoc_J} $F(M^\Tt_\alpha)$ is equivalent to the composition
of the extenders 
\[ \left<F\right>\conc\left<E^\Tt_\delta\right>_{\delta+1\in[\beta+1,\alpha]
_\Tt}\]
and for $\delta_0,\delta_1\in[\beta+1,\alpha]_\Tt$
with $\delta_0<\delta_1$, we have 
$\lambda(E^\Tt_{\delta_0})\leq\crit(E^\Tt_{\delta_1})$.
\end{enumerate}
\end{lem}

\begin{proof}[Proof sketch]

Part \ref{item:F_in_rev-e-dd}:
It is a standard fact (and easy to see)
that $M^{*\Tt}_{\beta+1}$
is in the mod-dropdown of $(M^\Tt_\gamma,E^\Tt_\gamma)$,
so this part follows easily from the definitions.

Part \ref{item:nu(F)_leq_kappa_<_lambda(F)}:
The fact that $\nu(F)\leq\crit(E^\Tt_\beta)<\lambda(E^\Tt_\gamma)$
is by definition,
and the fact that $\crit(E^\Tt_\beta)<\lambda(F)$ is standard.

Part \ref{item:ext_assoc_J} is as in \cite[Lemma 2.27***]{extmax}
(extended to transfinite iterations in a routine manner).
\end{proof}

\begin{dfn}
 Let $\Tt$ be a padded J-$m$-maximal tree on an $m$-sound J-premouse $M$.
 We say that $\Tt$ is \emph{nicely padded} iff
 there is an ordinal $\lambda>0$ and sequences 
$\left<\eta_\alpha\right>_{\alpha<\lambda}$,
$\left<n_\alpha\right>_{\alpha+1<\lambda}$ such that:
\begin{enumerate}
 \item $\eta_0=0$ and $\left<\eta_\alpha\right>_{\alpha<\lambda}$ is continuous
 and  $\Tt$ has length $\eta=\sup_{\alpha<\lambda}(\eta_\alpha+1)$.
 \item If $\alpha+1<\lambda$ then $n_\alpha<\om$ and
 $\eta_{\alpha+1}=\eta_\alpha+n_\alpha+1$.
 \item If $\eta_\alpha$ is $\Tt$-vs or $\alpha+1=\lambda$
 then $n_\alpha=0$.
 \item Let $\alpha+1<\lambda$ with
 $\eta_\alpha$ non-$\Tt$-vs. Then $n_\alpha>0$. Moreover, let
  $E=E^\Tt_{\eta_\alpha+n_\alpha}$.
  Then:
  \begin{enumerate}[label=--]
    \item $E\neq\emptyset$ and 
$\widetilde{\nu}^\Tt_{\eta_\alpha+n_\alpha}=\lambda(E)$.
  \item $E^\Tt_{\eta_\alpha+i}=\emptyset$
  and $\pred^\Tt(\eta_\alpha+i+1)=\eta_\alpha+i$ 
   for each $i< n_\alpha$ (so 
$M^\Tt_{\eta_\alpha}=M^\Tt_{\eta_\alpha+n_\alpha}$
and $\deg^\Tt_{\eta_\alpha}=\deg^\Tt_{\eta_\alpha+n_\alpha}$).

 \item The $\Tt$-despec-rev-e-dropdown of $(M^\Tt_{\eta_\alpha},E)$
has length $n_\alpha$; let it be
$\left<G^\Tt_{\eta_\alpha},G^\Tt_{\eta_\alpha+1},\ldots,G^\Tt_{
\eta_\alpha+n_\alpha-1}\right>$,
 so $G^\Tt_{\eta_\alpha+n_\alpha-1}=E$.
\item
$\widetilde{\nu}^\Tt_{\eta_\alpha+i}=\nu(G^\Tt_{\eta_\alpha+i})$
for each $i< n_\alpha$.\footnote{So if
$\nu(E)=\lambda(E)$ then 
$\widetilde{\nu}^\Tt_{\eta_\alpha+n_\alpha-1}=\nu(E)=\lambda(E)=
\widetilde{\nu}^\Tt_{\eta_\alpha+n_\alpha}$.
In this situation it would have been in some ways
 more natural to make $E^\Tt_{\eta_\alpha+n_\alpha-1}=E$
 instead of $E^\Tt_{\eta_\alpha+n_\alpha-1}=\emptyset$
 and $E^\Tt_{\eta_\alpha+n_\alpha}=E$,
 but for uniformity of notation it turns out simpler
 to keep the separation.}
\end{enumerate}
\end{enumerate}

Write $\eta^\Tt_\alpha=\eta_\alpha$, $n^\Tt_\alpha=n_\alpha$,
$n'^{\Tt}_\alpha=n'_\alpha$.

We say that  $(M,m,\Tt)$ is \emph{suitable}
iff $M$ is an $m$-sound J-premouse and $\Tt$ is
a nicely padded J-$m$-maximal tree on $M$.
\end{dfn}

\begin{dfn}
 Let $(M,m,\Tt)$ be suitable.
 Let $\alpha+1<\lh(\Tt)$ with $E^\Tt_\alpha\neq\emptyset$,
 so $\alpha=\eta_\xi+n_\xi$ for some $\xi$ (with notation as above).
If $\alpha$ is $\Tt$-vs and $\beta$ is least 
witnessing that $\alpha$ is $\Tt$-special,
then let $\widetilde{\alpha}$
denote $\pred^\Tt(\beta+1)-1$ (by Lemma \ref{lem:organization_specials} below,
 $\pred^\Tt(\beta+1)=\eta_\chi+i+1$ for some $\chi,i$),
and let $\core(E^\Tt_\alpha)$
denote $F(M^{*\Tt}_{\beta+1})$ (also
by that lemma,  $F(M^{*\Tt}_{\beta+1})=G^\Tt_{\eta_\chi+i}$).

If $\alpha$ is non-$\Tt$-vs (so $n_\xi>0$) then let 
$\widetilde{\alpha}=\eta_\xi+n_\xi-1$
and  
$\core(E^\Tt_\alpha)=E^\Tt_\alpha=G^\Tt_{\widetilde{\alpha}}$).
\end{dfn}

We make some basic observations on the interaction between
$\Tt$-special ordinals and nice pudding:
\begin{lem}\label{lem:organization_specials}
Let $(M,m,\Tt)$ be suitable and adopt notation as above.
Let $\eta+1<\lh(\Tt)$ with $E^\Tt_\eta\neq\emptyset$.
Then:
\begin{enumerate}
 \item Suppose $\pred^\Tt(\eta+1)=\eta_\alpha+i+1$. Then:
 \begin{enumerate}
\item $M^{*\Tt}_{\eta+1}\ins M^\Tt_{\eta_\alpha}|\lh(G^\Tt_{\eta_\alpha+i})$.
\item  $\eta+1$ is $\Tt$-special iff 
$M^{*\Tt}_{\eta+1}=M^\Tt_{\eta_\alpha}|\lh(G^\Tt_{\eta_\alpha+i})$.
\item If $\eta+1$ is $\Tt$-special
then:
\begin{enumerate}[label=--]
\item $\core(F(M^\Tt_{\eta+1}))=G^\Tt_{\eta_\alpha+i}$.
\item If  $\eta+1\notin\dropset^\Tt$
then $i=0$ and $G^\Tt_{\eta_\alpha}=F(M^\Tt_{\eta_\alpha})$
and $\eta_\alpha$ is non-$\Tt$-special.
\end{enumerate}
\end{enumerate}
\item Suppose $\pred^\Tt(\eta+1)=\eta_\alpha$. Then:
\begin{enumerate}
\item The following are equivalent:
\begin{enumerate}[label=(\roman*)]
\item\label{item:equiv_i} $\eta+1$ is $\Tt$-special, 
\item\label{item:equiv_ii} $\eta+1\notin\dropset^\Tt$
and $M^\Tt_{\eta_\alpha}$
is active and $\nu(F(M^\Tt_{\eta_\alpha}))\leq\crit(E^\Tt_\eta)$,
\item $\eta_\alpha$ is $\Tt$-special and $\eta+1\notin\dropset^\Tt$.
\end{enumerate}
\item If $\eta+1$ is $\Tt$-special
then $\core(F(M^\Tt_{\eta+1}))=\core(F(M^\Tt_{\eta_\alpha}))$.
\end{enumerate}
\end{enumerate}
\end{lem}
\begin{proof}[Proof sketch] Suppose $\pred^\Tt(\eta+1)=\eta_\alpha$.
A short unravelling of definitions shows \ref{item:equiv_i} $\Leftrightarrow$ 
\ref{item:equiv_ii}.
Suppose $\eta+1$
is $\Tt$-special; let us observe that $\eta_\alpha$
is $\Tt$-special. Let $F=F(M^\Tt_{\eta_\alpha})$.
By \ref{item:equiv_ii}, 
$\eta+1\notin\dropset^\Tt$ and $F\neq\emptyset$
and $\nu(F)\leq\crit(E^\Tt_\eta)$.
It follows easily that
$F$ is in the rev-e-dropdown
of $(M^\Tt_{\eta_\alpha},E^\Tt_{\eta_\alpha})$.
But then if $\eta_\alpha$ is non-$\Tt$-special,
the $\Tt$-despec-rev-e-dropdown is just the rev-e-dropdown,
so
$G^\Tt_{\eta_\alpha}=F$,
so $\widetilde{\nu}^\Tt_{\eta_\alpha}=\nu(F)\leq\crit(E^\Tt_\eta)$,
so $\pred^\Tt(\eta+1)>\eta_\alpha$, contradiction.\end{proof}

\begin{dfn}
Let $(M,m,\Tt)$ be suitable.
For $\gamma+1,\alpha+1<\lh(\Tt)$ such that
$E^\Tt_\gamma\neq\emptyset\neq E^\Tt_\alpha$,
write $\gamma\leq^{\ext,\Tt}_{\direct}\alpha$
iff either $\gamma=\alpha$ or $\alpha$ is $\Tt$-vs
and $\gamma+1\in(\widetilde{\alpha}+1,\alpha]_\Tt$.
Let ${\leq^{\ext,\Tt}}$ be the transitive closure
of ${\leq^{\ext,\Tt}}$.
Let ${<^{\ext,\Tt}_{\direct}}$ and ${<^{\ext,\Tt}}$
be the strict parts.

The \emph{standard decomposition} of $E^\Tt_\alpha$
 is the enumeration 
 of
 \[\{\core(E^\Tt_\beta)\bigm|\beta\leq^{\ext,\Tt}\alpha\},\]
 in order of increasing critical point.
\end{dfn}

 Note here that if $\beta_0<^{\ext,\Tt}\beta_1$ then
 $\beta_0<\beta_1$ but 
$\crit(\core(E^\Tt_{\beta_1}))<\crit(\core(E^\Tt_{\beta_0}))$.

\begin{lem}
 Let $(M,m,\Tt)$ be suitable and $\alpha+1<\lh(\Tt)$.
 
 The standard decomposition of $E^\Tt_\alpha$ 
 is well-defined. That is, if $\beta_0,\beta_1\leq^{\ext,\Tt}\alpha$
 and $\beta_0\neq\beta_1$ then $\kappa_{\beta_0}\neq\kappa_{\beta_1}$
 where $\kappa_{\beta_i}=\crit(\core(E^\Tt_{\beta_i}))=\crit(E^\Tt_{\beta_i})$;
 moreover, if $\beta_0=\alpha$ then $\kappa_{\beta_0}<\kappa_{\beta_1}$,
 and if $\kappa_{\beta_0}<\kappa_{\beta_1}$
 then $\nu(\core(E^\Tt_{\beta_0}))\leq\kappa_{\beta_1}$.

 Further,  $E^\Tt_\alpha$ is equivalent to composition of the
 extenders in the
  standard decomposition
  of
  $E^\Tt_\alpha$ (in order of increasing critical point).
\end{lem}

\begin{proof}
 By induction using Lemma \ref{lem:core_appears_J},
again as in \cite[Lemma 2.27***]{extmax}.
\end{proof}

\begin{dfn}
 Let $(M,m,\Tt)$ be suitable.
 Let $\alpha<\lh(\Tt)$.
Given $\gamma\leq^\Tt\alpha$,
 $\vec{E}^\Tt_{\gamma\alpha}$ denotes the sequence
$\left<E^\Tt_\beta\right>_{\gamma<^\Tt\beta+1\leq^\Tt\alpha}$
(so $\vec{E}^\Tt_{\gamma\alpha}$ corresponds to $i^\Tt_{\gamma\alpha}$
when the latter exists),
and $\vec{F}^\Tt_{\gamma\alpha}$ denotes
the sequence 
$\left<E^\Tt_\beta\right>_{\beta+1\in[\xi+1,\alpha]
_\Tt}$,
where $\xi$ is least such that
$\gamma<^\Tt\xi+1\leq^\Tt\alpha$ and $(\xi+1,\alpha]_\Tt$
does not drop in model.
Write $\vec{E}^\Tt_{<\alpha}=\vec{E}^\Tt_{0\alpha}$
and $\vec{F}^\Tt_{<\alpha}=\vec{F}^\Tt_{0\alpha}$.

Given a sequence $\vec{E}=\left<E_\alpha\right>_{\alpha<\lambda}$
of short extenders, we define
$\left<U_\alpha,k_\alpha\right>_{\alpha\leq\lambda}$,
if possible, by induction on $\lambda$, 
as follows. Set $U_0=M$ and $k_0=k$.
Given $U_\eta$  and $k_\eta\leq\om$
are well-defined and $U_\eta$ is a $k_\eta$-sound
seg-pm
and $\eta<\lambda$,
then: if $\crit(E_\eta)<\OR(U_\eta)$
and there is $N\ins U_\eta$ such that $E_\eta$ measures
exactly $\pow([\kappa]^{<\om})\inter N$,
then letting $N\ins U_\eta$ be the largest such,
and letting $n\leq\om$ be largest such that $(N,n)\ins(U_\eta,k_\eta)$
and $\crit(E_\eta)<\rho_{n}^N$,
if such $n$ exists,\footnote{It might be
that $N=U_\eta$ is a type 3 MS-premouse
with $\crit(E_\eta)=\lgcd(U_\eta)$,
in which case $n$ does not exist.}
 then $U_{\eta+1}=\Ult_{n}(N,E_\eta)$ and $k_{\eta+1}=n$.
We say there is a \emph{drop in model} 
at $\eta+1$ iff $N\pins U_\eta$.
Given a limit $\eta$ such that $U_\alpha$ is well-defined
for each $\alpha<\eta$,
then $U_\eta$ is well-defined iff there are only finitely many drops in model
${<\eta}$, and then $U_\eta$ is the resulting direct limit
and $k_\eta=\lim_{\alpha<\eta}k_\alpha$.
We  define $\Ult_{k}(M,\vec{E})=U_\lambda$,
if this is well-defined,
and if so, and there is no drop in model,
we define the \emph{iteration map} $i^{M,k}_{\vec{E}}:M\to U_\lambda$ 
resulting naturally from the ultrapower maps.
Also, if there is no drop in model,
or the only drop in model occurs at $1$,
then define $\bar{i}^{M,k}_{\vec{E}}:N\to U_\lambda$ where $N\ins M$
is as above.
\end{dfn}

\begin{rem}
Let $(M,m,\Tt)$ be suitable and $\beta\leq^\Tt\alpha<\lh(\Tt)$.
Then clearly
 $M^\Tt_\alpha=\Ult_{k}(M^\Tt_\beta,\vec{E}^\Tt_{\beta\alpha})$
 and 
$i^\Tt_{\beta\alpha}=i^{M^\Tt_\beta,k}
_{\vec{E}^\Tt_{\beta\alpha}}$ 
where $k=\deg^\Tt_\beta$ (with one map defined iff the other is),
and likewise 
$i^{*\Tt}_{\gamma+1,\alpha}=\bar{i}^{M^\Tt_\beta,k}_{\vec{E}^\Tt_{\beta\alpha}}
$ when $\pred^\Tt(\gamma+1)=\beta$.
\end{rem}

\begin{dfn}
 Let $(M,m,\Tt)$ be suitable.

 We say that $\Tt$ is \emph{unravelled}
 iff, if $\lh(\Tt)=\alpha+1$ then $\alpha$ is non-$\Tt$-special.
 The \emph{unravelling} $\unrvl(\Tt)$ of $\Tt$
 is the unique unravelled J-$m$-maximal tree $\Ss$ on $M$,
 if it exists, such that (i) $\Tt\ins\Ss$, (ii) if $\lh(\Tt)$ is a limit then 
$\Tt=\Ss$, and (iii) if $\lh(\Tt)=\alpha+1$ then $\beta$ is $\Ss$-vs 
for each $\beta+1<\lh(\Ss)$
with $\alpha\leq\beta$. (So
$E^\Ss_\beta=F(M^\Ss_\beta)$ for each $\beta\geq\alpha$,
and $\crit(E^\Ss_{\alpha+i+1})<\crit(E^\Ss_{\alpha+i})$,
so $\lh(\Ss)<\lh(\Tt)+\om$; the existence of $\Ss$ just depends on the 
wellfoundedness of the resulting models.)

We say that $\Tt$ is \emph{everywhere unravelable}
iff  (i) $\unrvl(\Tt\rest(\eta^\Tt_\alpha+1))$ exists 
for every 
$\alpha<\lambda^\Tt$, and (ii) for each $\alpha+1<\lambda^\Tt$ and 
$i<n^{\Tt}_\alpha$,
 letting \[\Ww=(\Tt\rest(\eta_\alpha+i+2))\conc 
G^\Tt_{\eta_\alpha+i}\] (a nicely padded
putative J-$m$-maximal tree)\footnote{That is,
it satisfies all the requirements
of a nicely padded J-$m$-maximal tree
excluding the wellfoundedness of the last model.
Note that we use $\Tt\rest(\eta_\alpha+i+2)$
(which has last model $M^\Tt_{\eta_\alpha+i+1}$)
as opposed to $\Tt\rest(\eta_\alpha+i+1)$ or
even $\Tt\rest(\eta_\alpha+1)$, because
although $G=G^\Tt_{\eta_\alpha+i}\in\es_+(M^\Tt_{\eta_\alpha})=
\es_+(M^\Tt_{\eta_\alpha+i})$,
using exit extender $G$
at stage $\eta_\alpha+i$ in $\Tt'$ would not
give a nicely padded tree $\Tt'$. By
Lemma \ref{lem:rev-e-dd_seg}, we do get a nicely padded 
tree by using it at stage $\eta_\alpha+i+1$.},
$\Ww$ has wellfounded models (so is actually
a nicely padded J-$m$-maximal tree) and $\unrvl(\Ww)$ exists
(another such tree).\footnote{When $\alpha'=\alpha+1$,
clause (i) for $\alpha'$ replacing $\alpha$
is just the same as the instance of clause (ii)
with  $\alpha$ and $i=n^\Tt_\alpha-1$, but clause
(i) for limit $\alpha$ is not covered by clause (ii).}
\end{dfn}

\begin{dfn}\label{dfn:conversion_J}
 Let $(M,m,\Tt)$ be suitable
 with  $\Tt$ unravelled
 and everywhere unravelable.
 The \emph{conversion} $\conv(\Tt)$
 is the unique padded lifting-MS-$m$-maximal tree $\Uu$
 on $M$, with exchange ordinals $\widetilde{\nu}^\Uu_\alpha$,
 satisfying (we verify later that this works, writing
 $\lambda=\lambda^\Tt$, $\eta_\alpha=\eta^\Tt_\alpha$, etc):
 \begin{enumerate}
  \item\label{item:tree_length_exchange_ord_agmt_J}$\lh(\Uu)=\lh(\Tt)$ 
and 
$\widetilde{\nu}^\Uu_\gamma=\widetilde{\nu}^\Tt_\gamma$
for all $\gamma+1<\lh(\Tt)$.
  \item If $\alpha+1<\lambda$ and $\eta_\alpha$ is $\Tt$-vs (so 
$n_\alpha=0$)
  then $E^\Uu_{\eta_\alpha}=\emptyset$,
  $\pred^\Uu(\eta_{\alpha+1})=\eta_\alpha$.
  \item If $\alpha+1<\lambda$ and $\eta_\alpha$
is non-$\Tt$-vs then $E^\Uu_{\eta_\alpha+i}=G^\Tt_{\eta_\alpha+i}$
  for all $i< n_\alpha$,
  and $E^\Uu_{\eta_\alpha+n_\alpha}=\emptyset$,
  with $\pred^\Uu(\eta_{\alpha+1})=\eta_\alpha+n_\alpha$.\footnote{So note
  $E^\Uu_{\eta_\alpha+n_\alpha-1}=E^\Tt_{\eta_\alpha+n_\alpha}$.}
  \item\label{item:conversion_limits_J} Let $\eta<\lh(\Tt)$ be a limit.
  Fix $\gamma<^\Tt\eta$
  such that $(\gamma,\eta)_\Tt\inter\dropset^\Tt=\emptyset$
  and if any $\xi\in(\gamma,\eta)_\Tt$
  is $\Tt$-special then $\gamma$ is $\Tt$-special.
  Let $X$ be the set of all $\beta<\eta$ such that 
$\beta\leq^{\ext,\Tt}\alpha$ for some
$\alpha+1\in(\gamma,\eta)_\Tt$ with $E^\Tt_\alpha\neq\emptyset$.
Then:
\begin{enumerate}
\item $\{\widetilde{\beta}+1\bigm|\beta\in X\}$ is cofinal in 
$\eta$;
 \item 
 $E^\Uu_{\widetilde{\beta}}=\core(E^\Tt_{\beta})$
 for each $\beta\in X$;
\item for all $\beta_0,\beta_1\in X$,
either $\widetilde{\beta_0}+1\leq^\Uu\widetilde{\beta_1}+1$ or
vice versa;
\item  $[0,\eta)_\Uu$ is the
$\leq^\Uu$-downward closure
of $\{\widetilde{\beta}+1\bigm|\beta\in X\}$.\qedhere
\end{enumerate}  
 \end{enumerate}
\end{dfn}

\begin{lem}\label{lem:conversion_exists_J}
  Let $(M,m,\Tt)$ be suitable with $\Tt$ unravelled and everywhere 
unravelable, and $\Tt$ non-trivial. Write $\lambda=\lambda^\Tt$ etc. Then:
\begin{enumerate}
 \item $\Uu=\conv(\Tt)$ is a well-defined padded MS-lifting-$m$-maximal
 tree on $M$.
\item Suppose $\lh(\Tt)=\eta+1$ (so $\eta$
is non-$\Tt$-special).
Let  $\varepsilon^\Tt+1\leq^\Tt\eta$ be least such that
$(\varepsilon^\Tt+1,\eta]_\Tt\inter\dropset^\Tt=\emptyset$,
and $\varepsilon^\Uu$ likewise for $\Uu$. Then (and let $\delta,N^*,k$ be 
defined 
by):
\begin{enumerate}
 \item $M^\Uu_\eta=M^\Tt_\eta$,
\item $[0,\eta]_\Tt\inter\dropset^\Tt=\emptyset\iff [0,\eta]
_\Uu\inter\dropset^\Uu=\emptyset$.
\item $\delta=\pred^\Tt(\varepsilon^\Tt+1)=\pred^\Uu(\varepsilon^\Uu+1)$,
\item  
$N^*=M^{*\Uu}_{\varepsilon^\Uu+1}=M^{*\Tt}_{\varepsilon^\Tt+1}$,
\item $\vec{F}^\Uu_{<\eta}$ is equivalent to 
$\vec{F}^\Tt_{<\eta}$, and in fact,
$\vec{F}^\Uu_{<\eta}$ is the enumeration of
\[\{\core(E^\Tt_\beta)\bigm|\exists\gamma\ 
\big[\varepsilon^\Tt+1\leq^\Tt\gamma+1\leq^\Tt\eta\text{ and } 
\beta\leq^{\exit,\Tt}\gamma\big]\} \]
in order of increasing critical point,
\item  $k=\deg^\Tt(\eta)=\deg^\Uu(\eta)$,
\item 
$M^\Tt_\eta=\Ult_{k}(N^*,\vec{F}^\Tt_{<\eta})=\Ult_{k}(N^*,\vec{F}^\Uu_{<\eta}
)=M^\Tt_\eta$,
\item $i^{*\Uu}_{\varepsilon^\Uu+1,\eta}=i^{*\Tt}_{\varepsilon^\Tt+1,\alpha}$.
\end{enumerate}

\item\label{item:unravelled_insegs} We have:
\begin{enumerate}
\item For  $\alpha<\lambda$,
\[\conv(\unrvl(\Tt\rest\eta_\alpha+1))=(\Uu\rest\eta_\alpha+1)\conc(\emptyset,
\ldots,\emptyset)
\]
(where $(\Uu\rest\eta_\alpha+1)\conc(\emptyset,\ldots,\emptyset)$
is an extension of $\Uu\rest\eta_\alpha+1$ by just padding;
the extension is finitely long), and
\item For each $\alpha+1<\lambda$ and $i< n_\alpha$,
\[ \conv\Big(\unrvl\Big((\Tt\rest\eta_\alpha+i+2)\conc 
G^\Tt_{\eta_\alpha+i})\Big)\Big)=
 (\Uu\rest\eta_\alpha+i+2)\conc(\emptyset,\ldots,\emptyset).
\]
\end{enumerate}
\item\label{item:comparability_J} Let $\eta+1<\lh(\Tt)$ with 
$E^\Tt_\eta\neq\emptyset$ and 
$X=\{\beta\bigm|
\beta\leq^{\ext,\Tt}\eta\}$.
Then for each $\beta\in X$,
we have $E^\Uu_{\widetilde{\beta}}=\core(E^\Tt_\beta)$,
and for all $\beta_0,\beta_1\in X$, if 
$\crit(E^\Tt_{\beta_0})\leq\crit(E^\Tt_{\beta_1})$
then
$\widetilde{\beta_0}+1\leq^\Uu\widetilde{\beta_1}+1\leq^\Uu\eta+1$
and
$(\widetilde{\beta_0}+1,\eta+1]
_\Uu\inter\dropset^\Uu_{\deg}=\emptyset$.
\item\label{item:iterated_comparability_J} Let $\eta+1<^\Tt\eta'+1<\lh(\Tt)$
be such that $E^\Tt_\eta\neq\emptyset\neq E^\Tt_{\eta'}$
and $\eta'+1$ is non-$\Tt$-special
and $(\eta+1,\eta'+1]_\Tt$ does not drop in model.
Let  $\beta\leq^{\ext,\Tt}\eta$
and $\beta'\leq^{\ext,\Tt}\eta'$.
Then:
\begin{enumerate}[label=--]
\item $\widetilde{\eta}+1\leq^\Uu\widetilde{\beta}
+1\leq^\Uu\eta+1\leq^\Uu
\widetilde{\eta'}+1\leq^\Uu\widetilde{\beta'}+1\leq^\Uu\eta'+1$,
\item $(\widetilde{\eta}+1,\eta'+1]_\Uu\inter\dropset^\Uu=\emptyset$, and
\item if
$(\eta+1,\eta'+1]_\Tt\inter\dropset_{\deg}^\Tt=\emptyset$
then 
$(\widetilde{\eta}+1,\eta'+1]_\Uu\inter\dropset_{\deg}^\Uu=\emptyset$.
\end{enumerate}
\end{enumerate}
\end{lem}
\begin{proof}
The proof is by induction on the unravelled trees
which appear in part \ref{item:unravelled_insegs}.
For $\lh(\Tt)=1$ it is trivial and for $\lh(\Tt)$ a limit,
it follows immediately by induction. So suppose $\lh(\Tt)=\varepsilon+1$
for some $\varepsilon>0$.

\begin{case} $\lh(\Tt)=\eta_\xi+n_\xi+2+n$
where  $n<\om$ and
$\eta_\xi$ is non-$\Tt$-vs but $\eta_\xi+n_\xi+1+i$ is
$\Tt$-vs for all $i<n$.

So $E=E^\Uu_{\eta_\xi+n_\xi-1}=E^\Tt_{\eta_\xi+n_\xi}\neq\emptyset$,
this is the last non-empty extender used in $\Uu$,
$\eta_{\xi+1}=\eta_\xi+n_\xi+1$,
and $\Tt=\unrvl(\Tt\rest(\eta_{\xi+1}+1))$.
Let $\mu\leq\xi$ be least such that $E^\Tt_\eta$ is $\Tt$-vs
for each 
$\eta\in[\eta_\mu,\eta_\xi)$. Let
\[ \bar{\Tt}=\unrvl(\Tt\rest(\eta_\mu+1)),\]
and say $\lh(\bar{\Tt})=\eta_\mu+\ell+1$ (so $\ell<\om$).
Let $\bar{\Uu}=\conv(\bar{\Tt})$.
So $\eta_\xi\in[\eta_\mu,\eta_\mu+\ell]$ 
and by induction, we have
\[ 
M^{\bar{\Tt}}_{\eta_\mu+\ell}=M^{\bar{\Uu}}_{\eta_\mu+\ell}=M^{\bar{\Uu}}
_{\eta_\mu}=M^\Uu_{\eta_\mu}=M^\Uu_{\eta_\xi}.\]
It follows that 
 $\lh(E^\Uu_{\eta_\alpha+n_\alpha})\leq\lh(G^\Tt_{\eta_\xi+n_\xi})$
for each $\alpha<\xi$ 
and  (by induction on $i$) $G^\Tt_{\eta_\xi+i}\in\es_+(M^\Uu_{\eta_\xi+i})$
for each $i\leq n_\xi$, and we can set 
$E^\Uu_{\eta_\xi+i}=G^\Tt_{\eta_\xi+i}$ for each $i\leq n_\xi$.
By induction, the lemma also holds for
\[ \unrvl\Big(\Tt\rest(\eta_\xi+i+2)\conc G^\Tt_{\eta_\xi+i}\Big) \]
for each $i+1<n_\xi$.
Let $\kappa=\crit(E)$
and
\[ 
\eta_\chi+j=\pred^\Tt(\eta_{\xi+1})=\pred^\Tt(\eta_\xi+n_\xi+1)=
\pred^\Uu(\eta_{\xi}+n_\xi), \]
(recall $\widetilde{\nu}^{\Tt}_\beta=\widetilde{\nu}^{\Uu}_\beta$ 
for all 
$\beta+1<\lh(\Tt)$,
by \ref{dfn:conversion_J}). Recall 
$E^\Uu_{\eta_\xi+n_\xi}=\emptyset$
 and
 $\pred^\Uu(\eta_{\xi+1})=\eta_\xi+n_\xi$.

\begin{scase} $\eta_{\xi+1}$ is non-$\Tt$-special.
 
 Let $N^*=M^{*\Tt}_{\eta_{\xi+1}}$ and 
$d=\deg^\Tt_{\eta_{\xi+1}}$.
 We claim that $M^{*\Uu}_{\eta_{\xi}+n_\xi+1}=N^*$
 and $\deg^\Uu_{\eta_\xi+n_\xi+1}=d$.
 For if $\eta_\chi$ is $\Tt$-special
 and $j=0$ then $\eta_{\xi+1}\in\dropset^\Tt$
 (as $\eta_{\xi+1}$ is non-$\Tt$-special),
 but then since $(M^\Tt_{\eta_\chi})^\passive\ins M^\Uu_{\eta_\chi}$
 and $\OR(M^\Tt_{\eta_\chi})$ is a cardinal of $M^\Uu_{\eta_\chi}$,
this gives the claim in this case. Suppose
 $\eta_\chi$ is non-$\Tt$-special or $j>0$.
If $j=0$ then
 $M^\Tt_{\eta_\chi}=M^\Uu_{\eta_\chi}$
 and $\deg^\Tt_{\eta_\chi}=\deg^\Uu_{\eta_\chi}$,
 which suffices. 
Suppose $j>0$
 and let $G=G^\Tt_{\eta_\chi+j-1}$, so
\[ \rho_1^{M^\Tt_{\eta_\chi}|\lh(G)}\leq\nu(G )\leq\kappa.\]
Now $\exit^\Tt_{\eta_\chi+n_\chi}\ins M^\Tt_{\eta_\chi}|\lh(G)$,
and
as $\eta_{\xi+1}$ is non-$\Tt$-special, therefore
$N^*\pins M^\Tt_{\eta_\chi}|\lh(G)$
(and in particular $\eta_{\xi+1}\in\dropset^\Tt$),
and by coherence, we have $N^*\pins M^\Uu_{\eta_\chi+j}$,
so $N^*=M^{*\Uu}_{\eta_\xi+n_\xi+1}$, 
which suffices.

 It follows that
$M^\Tt_{\eta_{\xi+1}}=M^\Uu_{\eta_{\xi+1}}$, and
(combined with induction if $\eta_{\xi+1}\notin\dropset^\Tt$,
noting that $j=0$ in this case)
there is appropriate
agreement of iteration maps.
(As $\eta_{\xi+1}$ is non-$\Tt$-special, $E$ is the last extender
of $\Tt''=\unrvl(\Tt\rest(\eta_{\xi+1}+1))$,
and $\conv(\Tt'')=\Uu\rest(\eta_{\xi+1}+1)$.)

The remaining properties in this subsubcase are now straightforward
to verify by induction.
\end{scase}

\begin{scase}\label{sscase:chi_non_T-special_is_trans_succ_J}
$\eta_{\xi+1}$ is $\Tt$-special and $j>0$.

By Lemma \ref{lem:organization_specials},
letting $G=G^\Tt_{\eta_\chi+j-1}$,
we have  $N^*=M^{*\Tt}_{\eta_{\xi+1}}=M^\Tt_{\eta_\chi}|\lh(G)$,
and $\widetilde{\nu}^\Tt_{\eta_\chi+j-1}=\nu(G)\leq\kappa<\lambda(G)$.
Also  $\deg^\Tt_{\eta_{\xi+1}}=0$.
So $M^\Tt_{\eta_{\xi+1}}=\Ult_0(N^*,E)$,
and note then that
\begin{equation}\label{eqn:ext_equiv} F(M^\Tt_{\eta_{\xi+1}})\text{
is equivalent to the two-step iteration }(G,E).\end{equation}

Let $\Tt'=\unrvl(\Tt\rest(\eta_\chi+j+1)\conc G)$
(nicely padded J-$m$-maximal), 
$\lh(\Tt')=\eta_\chi+j+2+\ell$
(so $\ell<\om$)
and 
\[ 
\Uu'=\conv(\Tt')=\Uu\rest(\eta_\chi+j+1)\conc(\emptyset)\conc(\emptyset,\ldots,
\emptyset),\]
noting the lemma applies to $\Tt'$ by induction
(note $E^{\Uu'}_{\eta_\chi+j-1}=G$
and $E^{\Uu'}_{\eta_\chi+j}=\emptyset$
and $\eta^{\Tt'}_{\chi+1}=\eta_\chi+j+1$).
Letting $k=\deg^{\Tt'}_\infty=\deg^{\Uu'}_\infty=\deg^\Uu(\eta_\chi+j)$, then
\[ 
M^{\Tt'}_\infty=M^{\Uu'}_\infty=M^\Uu_{\eta_\chi+j}
=\Ult_k(M^{*\Uu}_{\eta_\chi+j},G).
\]
Note  $\kappa=\crit(E)<\lambda(G)<\rho_k(M^\Uu_{\eta_\chi+j})$
and since $E$ is total over $M^\Tt_{\eta_\chi}|\lh(G)$
and by coherence etc, $E$ is total over $M^\Uu_{\eta_\chi+j}$.
So $\deg^\Uu_{\eta_{\xi+1}}=\deg^\Uu_{\eta_\xi+n_\xi}=k$
and
\[\begin{array}{rcccl}M^\Uu_{\eta_{\xi+1}}&=&M^\Uu_{\eta_\xi+n_\xi}&=
&\Ult_k(M^\Uu_{\eta_\chi+j},E)\\
&&&=&
\Ult_k(\Ult_k(M^{*\Uu}_{\eta_\chi+j},G),E)\\
&&&=&\Ult_k(M^{*\Uu}_{\eta_\chi+j},F(M^\Tt_{\eta_{\xi+1}})),\end{array} \]
with matching ultrapower maps,
by line (\ref{eqn:ext_equiv}).

Let $\Tt''=\unrvl(\Tt\rest(\eta_{\xi+1}+1))$
and
\[ 
\Uu''=\conv(\Tt'')=\Uu\rest(\eta_{\xi+1}+1)\conc(\emptyset,\ldots,\emptyset).\]
Then $\eta_{\xi+1}$ is $\Tt$-special, hence $\Tt''$-special,
\[ M^{\Tt''}_{\eta_{\xi+1}}=\Ult_0(M^\Tt_{\eta_\chi}|\lh(G),E) \]
and $F(M^{\Tt''}_{\eta_{\xi+1}})$ is equivalent to the two-step
iteration  $(G,E)$,
and then an easy induction gives that for 
each $i\in(0,\ell]$, $\eta_{\xi+1}+i$ is $\Tt''$-special and
\[ M^{\Tt''}_{\eta_{\xi+1}+i}=\Ult_{0}(M^{\Tt'}_{\eta_\chi+j+i},E) \]
and $F(M^{\Tt''}_{\eta_{\xi+1}+i})$ is equivalent to the two-step
iteration $(F(M^{\Tt'}_{\eta_\chi+j+i}),E)$.
Further, $\lh(\Tt'')=\eta_{\xi+1}+2+\ell$,
and recalling $k=\deg^{\Tt'}_\infty=\deg^{\Tt'}_{\eta_\chi+j+1+\ell}=
\deg^{\Uu'}_\infty$,
note $k=\deg^{\Tt''}_\infty=\deg^{\Tt''}_{\eta_{\xi+1}+1+\ell}$,
and
(letting) $P^*=M^{*\Tt'}_\infty=M^{*\Tt''}_\infty$,
we have
\[\begin{array}{rcrclclcl}M^{\Tt''}_\infty&=&
M^{\Tt''}_{\eta_{\xi+1}+1+\ell}&=&
\Ult_{k}(P^*,
F(M^{\Tt''}_{\eta_{\xi+1}+\ell}))\\
&&&=&\Ult_k(\Ult_k(P^*,
F(M^{\Tt'}_{\eta_\chi+j+\ell})),E)\\
&&&=&\Ult_k(M^{\Tt'}_{\eta_\chi+j+1+\ell},E)\\
&&&=&\Ult_k(M^{\Uu'}_{\eta_\chi+j},E)&=&M^{\Uu''}_\infty,\end{array}\]
with corresponding ultrapower maps,
 and since by induction
the iteration maps of $\Tt',\Uu'$ match appropriately,
so do those of $\Tt'',\Uu''$.

Regarding part \ref{item:comparability_J} for $\Tt''$ and
$X=\{\beta\bigm|\beta\leq^{\ext,\Tt}\eta_{\xi+1}\}$,
note
$X=\{\eta_\xi+n_\xi,\eta_{\xi+1}\}$
and $\widetilde{\eta_\xi+n_\xi}=\eta_\xi+n_\xi-1$
and $\widetilde{\eta_{\xi+1}}=\eta_\chi+j-1$,
and $E^{\Uu''}_{\widetilde{\gamma}}=\core(E^\Tt_\gamma)$
for $\gamma\in X$ (these two extenders are $G$ and $E$),
and 
\[ \eta_\chi+j\leq^{\Uu''}\eta_\xi+n_\xi<^{\Uu''}\eta_{\xi+1}<^{\Uu''}
\eta_{\xi+1}+1+i \]
for $i\leq\ell$, and
$(\eta_\chi+j,\eta_{\xi+1}+1+\ell]_{\Uu''}\inter\dropset_{\deg}^{\Uu''}
=\emptyset$ ($\Uu''$ only pads
after $\eta_\xi+n_\xi$).
Part \ref{item:iterated_comparability_J}
follows from the above
considerations and by induction applied to $\Tt'$.

\end{scase}
\begin{scase}\label{sscase:chi_T-special_J} $\eta_{\xi+1}$
is $\Tt$-special and $j=0$.

By Lemma 
\ref{lem:organization_specials},
$\eta_\chi$ is $\Tt$-special,
$\eta_{\xi+1}\notin\dropset^\Tt$,
 and $\nu(F)\leq\kappa<\lambda(F)$
 where $F=F(M^\Tt_{\eta_\chi})$.
Note then that $F(M^\Tt_{\eta_{\xi+1}})$ is
equivalent
to the two-step iteration $(F,E)$. So
things are almost the same as in Subcase 
\ref{sscase:chi_non_T-special_is_trans_succ_J},
with $F$ replacing $G$,
and we leave the details to the reader.

\end{scase}

\end{case}

\begin{case}
 $\lh(\Tt)=\eta_\zeta+\ell+1$ where $\zeta$ is a limit
 and $E^\Tt_{\eta_\zeta+i}$ is $\Tt$-vs for all $i<\ell$.
 
 Let $b=[0,\eta_\zeta)_\Tt$. Note that $\eta_\zeta$ is $\Tt$-special
iff $\eta_\alpha$ is $\Tt$-special
for all sufficiently large $\eta_\alpha\in b$.
By induction with parts \ref{item:comparability_J}
and
\ref{item:iterated_comparability_J}, $b$ induces
 a ${\Uu\rest\lambda}$-cofinal branch,
 which has the properties required by
 Definition \ref{dfn:conversion_J}(\ref{item:conversion_limits_J}).
 (If $\lambda$ is $\Tt$-special
 then apply part \ref{item:comparability_J}
 to  $\unrvl(\Tt\rest(\eta_\alpha+1))$ for sufficiently large 
$\eta_\alpha+1<^\Tt\eta_\zeta$.)

So if $\eta_\zeta$ is non-$\Tt$-special,
then induction easily shows that
$M^\Tt_{\eta_\zeta}=M^\Uu_{\eta_\zeta}$ and iteration maps agree 
appropriately
etc.
If $\eta_\zeta$ is $\Tt$-special, then proceed essentially
as in Subcase \ref{sscase:chi_T-special_J},
but using  $\vec{F}^\Tt_{\eta_\alpha\eta_\zeta}$
and the equivalent $\vec{F}^\Uu_{\eta_\alpha\eta_\zeta}$,
where $\eta_\alpha\in b$ is sufficiently large,
in place of single extenders of $\Tt,\Uu$.
\end{case}

This completes the proof of the lemma.
\end{proof}

\begin{lem}\label{lem:conversions_cover_J}
 Let $M$ be an $m$-sound J-premouse
  and $\Uu'$ be a lifting-MS-$m$-maximal 
tree on $M$. Then there is a unique pair $(\Tt,\Vv)$
such that $(M,m,\Tt)$ is suitable,
$\Tt$ is unravelled everywhere unravelable, 
  $\Vv=\conv(\Tt)$ and $\Uu'$ is the tree given by removing
  all padding from $\Vv$.
\end{lem}
\begin{proof}
We proceed by induction on $\lh(\Uu')$.
The induction is an easy consequence of Lemma \ref{lem:conversion_exists_J}
except for the case that $\lh(\Uu')=\zeta+1$ for some limit $\zeta$.
So consider this case,
assuming that the lemma holds for trees $\Uu''$
of length ${\leq\zeta}$.
So we have  trees 
$\Tt\rest\zeta$
and $\Uu\rest\zeta$ corresponding to $\Uu'\rest\zeta$,
where writing $\lambda=\lambda^{\Tt\rest\zeta}$, 
$\eta_\alpha=\eta_\alpha^{\Tt\rest\zeta}$ etc, (note) 
$\zeta=\lambda=
\sup_{\alpha<\zeta}\eta_\alpha$.
Write  $\eta_\zeta=\zeta$.
Set $\Uu=(\Uu\rest\eta_\zeta)\conc b$,
where $b$ is the branch determined by $[0,\zeta)_{\Uu'}$.
(So the desired $\Vv$ will be of the form 
$\Uu\conc(\emptyset,\ldots,\emptyset)$.)

\begin{clm}\label{clm:ev_stable_J}
 There is $\alpha<^\Uu\eta_\zeta$
 such that for all $\xi<\zeta$ and $i<n_\xi$
 with
 \[ \alpha<^\Uu\eta_\xi+i+1<^\Uu\eta_\zeta,\]
  letting $\delta=\lambda(G^\Tt_{\eta_\xi+i})=\lambda(E^\Uu_{\eta_\xi+i})$,
  there is $\chi\in[\eta_\xi+i+1,\eta_\zeta)_\Uu$
with
\[i^\Uu_{\eta_\xi+i+1,\chi}(\delta)\leq\crit(i^\Uu_{\chi\eta_\zeta}).\]
\end{clm}
\begin{proof}
If not, then select a sequence $\left<(\xi_n,i_n,\delta_n)\right>_{n<\om}$
of witnessing triples $(\xi,i,\delta)$ such that
writing $\gamma_n=\eta_{\xi_n}+i_n$, we have 
$\gamma_n+1<^\Uu\gamma_{n+1}+1$.
Then 
since 
\[ \delta_n= i^{*\Uu}_{\gamma_n+1}(\crit(E^\Uu_{\gamma_n})),\]
we get 
$i^\Uu_{\gamma_n+1,\eta_\zeta}(\delta_n)
>i^\Uu_{\gamma_{n+1}+1,\eta_\zeta}(\delta_{n+1})$
for each $n<\om$, so $M^\Uu_{\eta_\zeta}$ is illfounded, a contradiction.
\end{proof}

 We now break into cases, mostly 
in order that we can discuss a simpler case first as a warm-up:

\begin{casetwo}\label{case:not_cof_many_transitions_J}
For all sufficiently large $\eta+1<^\Uu\eta_\zeta$
with $E^\Uu_\eta\neq\emptyset$, we have $i^\Uu_{\eta+1,\eta_\zeta}$ exists and
$\lambda(E^\Uu_\eta)\leq\crit(i^{\Uu}_{\eta+1,\eta_\zeta})$.

Note then that for all sufficiently large $\eta+1<^\Uu\eta_\zeta$
with $E^\Uu_\eta\neq\emptyset$,
there is 
$\alpha$ such that
$\pred^\Uu(\eta+1)=\eta_\alpha$
(as when $\pred^\Uu(\eta+1)=\eta_\alpha+j$ with $j\in(0,n_\alpha]$,
and $i^\Uu_{\eta_\alpha+j,\eta_\zeta}$ exists,
then 
$\crit(i^\Uu_{\eta_\alpha+j,\eta_\zeta})<\lambda(E^\Uu_{\eta_\alpha+j-1})$).

Fix $\xi_0$
such that:
\begin{enumerate}[label=--]
 \item $\eta_{\xi_0}<^\Uu\eta_\zeta$ and 
$[\eta_{\xi_0},\eta_\zeta)_\Uu\inter\dropset_{\deg}^\Uu=\emptyset$
\item  $\pred^\Uu(\eta+1)=\eta_\alpha$ for some $\alpha$
whenever 
$\eta+1\in[\eta_{\xi_0},\eta_\zeta)_\Uu$
with $E^\Uu_\eta\neq\emptyset$
\end{enumerate}
(but don't demand that $\xi_0$ is least such). Let 
$b=[\eta_{\xi_0},\eta_\zeta)_\Uu$.

Let
$\eta+1\in b$
with $E^\Uu_\eta\neq\emptyset$,
and $\eta_\chi=\pred^\Uu(\eta+1)$.
Then $\eta=\eta_\xi+n_\xi-1$ for some $\xi$,
so $E^\Tt_\eta=\emptyset$, $E^\Tt_{\eta+1}=E^\Uu_\eta$, 
$E^\Uu_{\eta+1}=\emptyset$
and $\pred^\Uu(\eta+1)=\eta=\pred^\Tt(\eta+2)$, and note that 
by choice of 
$\xi_0$,
\[ \successor^\Uu(\eta+1,\eta_\zeta)=\eta_{\xi+1}=\eta+2.\]
Moreover, $\eta+2\notin\dropset_{\deg}^\Tt$.
For because
$\Uu\rest\eta_\zeta=\conv(\Tt\rest\eta_\zeta)$, if $\eta_\chi$ is 
non-$\Tt$-special then $M^\Tt_{\eta_\zeta}=M^\Uu_{\eta_\zeta}$ and
$\deg^\Tt_{\eta_\zeta}=\deg^\Uu_{\eta_\zeta}$,
and since $\eta+1\notin\dropset^\Uu_{\deg}$,
therefore $\eta+2\notin\dropset^\Tt_{\deg}$.
And if $\eta_\chi$ is $\Tt$-special
then $(M^\Tt_{\eta_\chi})^\passive$
is a cardinal segment of $M^\Uu_{\eta_\chi}$
and $\deg^\Tt_{\eta_\chi}=0$, and so
$\eta+2\notin\dropset^\Tt_{\deg}$.
Note that also $\lambda(E^\Uu_\eta)\leq\crit(i^\Uu_{\eta+1,\eta_\zeta})$
in this situation.

Let $\chi$ be least such that either $\chi=\eta_\zeta$
or  $\eta_{\xi_0}$
is $\Tt$-special and $\eta_{\xi_0}\leq^\Uu\chi<^\Uu\eta_\zeta$
and letting $\delta=\lgcd(M^\Tt_{\eta_{\xi_0}})$,
 we have
$i^\Uu_{\eta_{\xi_0}\chi}(\delta)\leq\crit(i^\Uu_{\chi\eta_\zeta})$
and $E^\Uu_\eta\neq\emptyset$ where 
$\eta+1=\successor^\Uu(\chi,\eta_\zeta)$.
Let $b'_0$ be the ${<\eta_\zeta}$-closure of
\[ 
[0,\eta_{\xi_0}]_\Tt\cup\Big\{
\eta+2\in b\inter(\chi+1)\Bigm|
 E^\Uu_\eta\neq\emptyset\Big\}.\]

\begin{clm}
$b'_0$ is a branch of $\Tt\rest\eta_\zeta$ (cofinal iff $\chi=\eta_\zeta$), and
$(b'\cut(\eta_{\xi_0}+1))\inter\dropset^\Tt_{\deg}=\emptyset$.\end{clm}
\begin{proof}We observe that this holds for the initial segments of 
$b'_0$,
by induction. For 
$b'_0\inter(\eta_{\xi_0}+1)$
it is trivial. Suppose it holds for $b'_0\inter(\beta+1)$
where $\beta\in b'_0$. Note that $\beta=\eta_\gamma$
for some $\gamma$, and $\eta_\gamma$
is $\Tt$-special iff $\eta_{\xi_0}$ is $\Tt$-special.
Suppose $\eta_\gamma<\chi$ and 
let
 $\eta+1=\successor^\Uu(\eta_\gamma,\eta_\zeta)$.

 Suppose $\eta$ is $\Tt$-vs. Then $\eta_\gamma=\pred^\Uu(\eta+1)=\eta$,
so $\eta_\gamma,\eta_{\xi_0}$ are $\Tt$-special, and
\[ 
\widetilde{\nu}^\Uu_{\eta}=\lambda(E^\Tt_\eta)=i^\Tt_{\eta_{\xi_0}
\eta}(\delta)=i^\Uu_{\eta_{\xi_0}\eta}(\delta)\leq
\crit(i^\Uu_{\eta+1,\eta_\zeta})=\crit(i^\Uu_{\eta\eta_\zeta}),\]
so $\chi\leq\eta$, contradiction. So $\eta$ is non-$\Tt$-vs.
But then $E^\Uu_\eta\neq\emptyset$ and by the remarks above,
$\pred^\Tt(\eta+2)=\eta$ etc, and the claim holds
for $b'_0\inter((\eta+2)+1)$.

Now let $\gamma>\xi_0$ be a limit
with $\eta_\gamma\in b$ and suppose the claim
holds below $\eta_\gamma$.
Then $[0,\eta_\gamma)_\Uu$
is determined from $[0,\eta_\gamma)_\Tt$
by Definition \ref{dfn:conversion_J},
and by induction, it easily follows
that $[0,\eta_\gamma]_\Tt=b'_0\inter(\eta_\gamma+1)$.
\end{proof}

Now if $\chi=\eta_\zeta$, so $b'_0$ is $\Tt\rest\eta_\zeta$-cofinal,
then set $\Tt'=\Tt\rest\lambda\conc b'_0$,  a well-defined
putative tree.

Suppose now that $\eta_\gamma=\chi<\eta_\zeta$ and let 
$\eta+1=\successor^\Uu(\eta_\gamma,\eta_\zeta)$.
Arguing as above, it follows that $\eta_\gamma$ is $\Tt$-vs,
 $E^\Uu_{\eta_\gamma}=\emptyset$
 $\pred^\Uu(\eta_\gamma+1)=\eta_\gamma$
 and $\eta_{\gamma+1}=\eta_\gamma+1<^\Uu\eta_\zeta$.
 Set $\xi_1=\gamma+1$, and define
 $b_1'$ from $\xi_1$ like $b_0'$ was defined from $\xi_0$.
 Proceed in this way defining $\xi_n,b_n'$ as far as possible.
 
 For $\alpha<\lh(\Uu)$, let $\ell_\alpha$
 be the $\ell<\om$ such that $\unrvl(\Tt\rest(\eta_\alpha+1))$
 has length $\eta_\alpha+1+\ell$.
 
Note that $\ell_\alpha$ is constant over
 all $\eta_\alpha\in b_n'$, and if $\eta_\alpha\in b_n'$ and $\eta_\beta\in 
b_{n+1}'$ then $\ell_{\beta}=\ell_\alpha-1$.
(This uses the properties of $b_n'$, including that
it is a branch of $\Tt$.)
So we reach $n<\om$ such that $b_n'$ is indeed $\Tt\rest\eta_\zeta$-cofinal,
and define $\Tt'=\Tt\rest\eta_\zeta\conc b_n'$. Set $b'=b_n'$ and
$\ell=\ell_\beta$ for $\eta_\beta\in b'$.

\begin{clm}\label{clm:unrvl(Tt')_at_limit_J}$\Tt'$  has wellfounded last model,
and moreover, $\unrvl(\Tt')$
is well-defined with wellfounded models,
and
 \begin{equation}\label{eqn:conv_unrvl_T'_J}\conv(\unrvl(\Tt'))= 
\Uu\rest(\eta_\zeta+1)\conc(\emptyset,\ldots,\emptyset) \end{equation}
and $\ell_{\zeta}^{\Tt'}=\ell$, i.e.
  the tuple $(\emptyset,\ldots,\emptyset)$
in (\ref{eqn:conv_unrvl_T'_J}) has length $\ell$.\end{clm}
\begin{proof}
Suppose $\ell=0$. Then $\eta_\beta$ is non-$\Tt$-special
for $\eta_\beta\in b'$, so note that $\Tt'$ is already unravelled,
and   $M^{\Tt'}_{\eta_\zeta}=M^{\Uu}_{\eta_\zeta}$.
So the fact that (\ref{eqn:conv_unrvl_T'_J}) holds in this case
follows easily by induction
and because our definition of $b'$ ensures
that $[0,\eta_\zeta)_\Uu$ is recovered from 
$b'$ via
Definition \ref{dfn:conversion_J}.

If $\ell>0$ then $\eta_\zeta$
is $\Tt'$-special, but it is similar,
with computation much as in the proof of 
Lemma \ref{lem:conversion_exists_J}.
(We have that $(M^\Tt_{\eta_{\xi_n}})^\passive$
is a cardinal segment of $M^\Uu_{\eta_{\xi_n}}$
and $\OR(M^\Tt_{\eta_{\xi+n}})\leq\rho_k(M^\Uu_{\eta_{\xi_n}})$
where $k=\deg^\Uu_{\eta_{\xi_n}}$,
so note that $(M^{\Tt'}_{b'})^\passive$
is a cardinal segment of $M^\Uu_{\eta_\zeta}$,
and in particular is wellfounded.
Now letting $\Tt''=\unrvl(\Tt')$,
proceed by induction on $i<\ell$,
rearranging extenders as in the proof of Lemma \ref{lem:conversion_exists_J},
to show that $(M^{\Tt''}_{\eta_\zeta+i})^\passive$
is a cardinal segment of $M^\Uu_{\eta_\zeta}$,
and finally that $M^{\Tt''}_\infty=M^\Uu_{\eta_\zeta}$
and $\deg^{\Tt''}_{\eta_\zeta+\ell}=\deg^{\Uu}_{\eta_\zeta}$.)
\end{proof}
\end{casetwo}

\begin{casetwo}
Otherwise: There are cofinally many $\eta+1<^\Uu\eta_\zeta$
such that $E^\Uu_\eta\neq\emptyset$ and $i^\Uu_{\eta+1,\eta_\zeta}$
exists and $\crit(i^\Uu_{\eta+1,\eta_\zeta})<\lambda(E^\Uu_\eta)$.

This case is  an embellishment of the previous one. Note that if 
$\eta+1<^\Uu\eta_\zeta$
and $E^\Uu\neq\emptyset$ and
and $i^\Uu_{\eta+1,\eta_\zeta}$ exists and
$\gamma+1=\successor^\Uu(\eta+1,\eta_\zeta)$,
then $\crit(i^\Uu_{\eta+1,\eta_\zeta})<\lambda(E^\Uu_\eta)$
iff $\eta=\eta_\alpha+j-1$ for some $\alpha$ and $j\in(0,n_\alpha]$
and it is not the case that $\gamma=\eta+1=\eta_\alpha+n_\alpha$.

Appealing to Claim \ref{clm:ev_stable_J} for existence, fix any $\xi_0$
such that:
\begin{enumerate}[label=(\alph*)]
 \item $\eta_{\xi_0}<^\Uu\eta_\zeta$ and 
$(\eta_{\xi_0},\eta_\zeta)_\Uu\inter\dropset_{\deg}^\Uu=\emptyset$,
\item\label{item:succ_xi_0_non_empty}$E^\Uu_\eta\neq\emptyset$ where
$\eta+1=\successor^\Uu(\eta_{\xi_0},\eta_\zeta)$,
\item 
for all $\eta+1\in(\eta_{\xi_0},\eta_\zeta)_\Uu$
with $E^\Uu_\eta\neq\emptyset$, there is $ \chi\in(\eta+1,\eta_\zeta)_\Uu$ such 
that
$i^\Uu_{\eta+1,\chi}(\lambda(E^\Uu_{\eta}))\leq\crit(i^\Uu_{\chi\eta_\zeta}
)$.
\end{enumerate}

For $\eta+1\in(\eta_{\xi_0},\eta_\zeta)$ with $E^\Uu_\eta\neq\emptyset$,
let $\chi_\eta$ be the least witness $\chi$ as above and such that
$E^\Uu_\gamma\neq\emptyset$ where $\gamma+1=\successor^\Uu(\chi,\eta_\zeta)$.
Note that if $\eta<\eta'$ are such
and $\eta'+1\leq\chi_\eta$ then $\chi_{\eta'}\leq\chi_\eta$.

Now let $\chi\in[\eta_{\xi_0},\eta_\zeta]_\Uu$ be least such that
either  $\chi=\eta_\zeta$
or 
\begin{enumerate}[label=--]
\item $\eta_{\xi_0}$ is $\Tt$-special
and
$i^\Uu_{\eta_{\xi_0}\chi}(\delta)\leq\crit(i^\Uu_{\chi\eta_\zeta})$
where $\delta=\lgcd(M^\Tt_{\eta_{\xi_0}})$ 
(so actually $\chi>\eta_{\xi_0}$, by \ref{item:succ_xi_0_non_empty}), and
\item $E^\Uu_\eta\neq\emptyset$
where $\eta+1=\successor^\Uu(\chi,\eta_\zeta)$.
\end{enumerate}

Note that if $\chi<\eta_\zeta$
then $\chi_\eta\leq\chi$ for all 
$\eta+1\in(\eta_{\xi_0},\chi]_\Uu$
with $E^\Uu_\eta\neq\emptyset$.

Let $b'_0=[0,\eta_{\xi_0}]_\Tt\cup b'$
where $b'$
is the ${<\chi}$-closure
of
\[ \chi\inter\Big\{\chi_\eta\Bigm|
\eta+1\in(\eta_{\xi_0},\eta_\zeta)_\Uu\text{ and }
E^\Uu_\eta\neq\emptyset\Big\}.\]
Note $b'_0\cut\eta_{\xi_0}\sub[\eta_{\xi_0},\eta_\zeta)_\Uu$.

\begin{clm} $b'_0$ is a branch of $\Tt\rest\lambda$
and $b'_0\cut(\eta_{\xi_0}+1)\inter\dropset^\Tt_{\deg}=\emptyset$.
\end{clm}
\begin{proof}
We establish the claim regarding $b'_0\inter(\beta+1)$
by induction on $\beta\in b'_0$ with $\beta\geq\eta_{\xi_0}$.
If $\beta=\eta_{\xi_0}$ it is trivial. 
Suppose the claim holds
for $b'_0\inter(\beta+1)$
where $\beta\in b'_0$ and $\beta\geq\eta_{\xi_0}$
and let $\eta+1=\successor^\Uu(\beta,\eta_\zeta)$.
We must see that $\chi_\eta$ is a successor ordinal,
$\pred^\Tt(\chi_\eta)=\beta$
and $\chi_\eta\notin\dropset^\Tt_{\deg}$.
If $\crit(i^\Uu_{\eta+1,\eta_\zeta})\geq\lambda(E^\Uu_\eta)$
then we just get $\chi_\eta=\eta+2$ (note $E^\Uu_{\eta+1}=\emptyset$)
and things are easy, so suppose 
$\crit(i^\Uu_{\eta+1,\eta_\zeta})<\lambda(E^\Uu_\eta)$.
Let $\chi'\leq\chi_\eta$ be least such that
$\chi'=\eta_\gamma$ for some $\gamma$
and 
$i^\Uu_{\eta+1,\chi'}(\lambda(E^\Uu_\eta))\leq\crit(i^\Uu_{\chi'\eta_\zeta})$,
so in fact $M^\Uu_{\chi'}=M^\Uu_{\chi_\eta}$,
and $\alpha$ is $\Tt$-vs for each $\alpha\in[\chi',\chi_\eta)$.
Let $\Tt'=\unrvl(\Tt\rest(\chi'+1))$,
so $\eta_\chi+1\leq\lh(\Tt')$
and $\Tt'\rest(\eta_\chi+1)\ins\Tt$.
Let $\Uu'=\conv(\Tt')$,
so $M^{\Uu'}_\infty=M^\Uu_{\chi'}=M^\Uu_{\chi_\eta}$
etc. Now by Lemma \ref{lem:conversion_exists_J}
applied to $\Tt',\Uu'$,  $\vec{F}^{\Tt'}_{<\infty}$
is equivalent to $\vec{F}^{\Uu'}_{<\infty}$,
and
\begin{equation}\label{eqn:key_concatenation} 
E^\Uu_\eta\conc\left<E^\Uu_\alpha\right>_{\alpha+1\in(\eta+1,\chi']} 
\end{equation}
is a tail segment of the latter.
But then the choice of $\chi_\eta$
easily gives that there is $\alpha\geq\chi'$ which is $\Tt$-vs
and $E^\Tt_\alpha$ is equivalent to the concatenation
in \ref{eqn:key_concatenation},
in particular with $\eta=\widetilde{\alpha}$
(computed relative to $\Tt',\Uu'$)
and $E^\Uu_\eta=\core(E^{\Tt'}_\alpha)$.
Therefore $\pred^{\Tt'}(\alpha+1)=\pred^\Uu(\eta+1)=\beta$,
and it is straightforward to see that $\alpha+1\notin\dropset^{\Tt'}_{\deg}$.
Finally, we claim that $\alpha+1=\chi_\eta$
and $\Tt'\rest(\chi_\eta+1)\ins\Tt$.
For if $\chi_\eta<\alpha+1$
then just note that 
$\crit(i^\Uu_{\chi_\eta,\eta_\zeta})<i^\Uu_{\eta+1,\chi_\eta}
(\lambda(E^\Uu_\eta)$,
and if $\alpha+1<\chi_\eta$ then
$\eta_{\xi_0}$ is $\Tt$-special
and 
$i^\Uu_{\eta_{\xi_0}\chi_\eta}(\delta)\leq\crit(i^\Uu_{\chi_\eta\eta_\zeta})$,
and then it easily follows that $\chi_\eta=\chi$,
contradicting that $\chi_\eta<\chi$ (by definition of $b'_0$).
So $\chi_\eta=\alpha+1$,
and it follows that $\Tt'\rest(\chi_\eta+1)\ins\Tt$.

Now let $\gamma$ be a limit with $\eta_\gamma\leq\chi$
and $\eta_\gamma\in(\eta_{\xi_0},\eta_\zeta]_\Uu$,
and suppose we have the claim below $\gamma$.
The fact that $[0,\eta_\gamma]_\Tt$ determines
$[0,\eta_\gamma)_\Uu$ as in Definition \ref{dfn:conversion_J},
together with induction, easily implies
that $[0,\eta_\gamma)_\Tt=b'_0\inter\eta_\gamma$.
\end{proof}

Now suppose that $\chi<\eta_\zeta$.
Then $\eta_{\xi_0}$ is $\Tt$-special,
and letting $\delta=\lgcd(M^\Tt_{\eta_{\xi_0}})$,
arguing as above (and almost by the claim),
we get $\chi'<^\Uu\chi$
such that $M^{\Uu}_{\chi'}=M^\Uu_\chi$
and $\eta_{\xi_0}\leq^\Tt\chi'$
and $\alpha$ is $\Tt$-vs for each $\alpha\in[\chi',\chi)$.
Hence $\chi'=\eta_\alpha$ for some $\alpha$.
Define $\ell_\alpha$ like before (so $\unrvl(\Tt\rest(\eta_\alpha+1))$
has length $\eta_\alpha+1+\ell_\alpha$).
Then $\ell_{\xi_0}=\ell_{\chi'}>\ell_\chi$.
Now let $\xi_1=\chi$. Given $\xi_n$, define $b_n'$ from $\xi_n$
like $b_0'$ was defined from $\xi_0$.
Then there is $n<\om$ with $b_n'$ cofinal in $\eta_\zeta$.

The rest is just like Claim \ref{clm:unrvl(Tt')_at_limit_J}
of Case \ref{case:not_cof_many_transitions_J}.
\end{casetwo}

This completes the proof of the lemma.
\end{proof}

\section{Background constructions for J-mice}

In this section we describe a background construction for J-mice
which is based on the traditional kind for MS-mice,
and absorbs Woodin cardinals from the background universe
just like traditionally for MS-mice. Other than being
based on those for MS-mice, it however incorporates
some features analogous to the construction of \cite[\S5]{premouse_inheriting}. 
(One could probably also go further in that direction,
to obtain a background construction which also absorbs
strong cardinals from the background.)

\begin{dfn}
 Let $M$ be a J-pm. We say that $M$ satisfies the \emph{MS-ISC}
 iff either $M$ is passive, or letting $\kappa=\crit(F^M)$,
 then for every $\nu\in[(\kappa^+)^M,\nu(F^M))$
 such that $\nu$ is the natural length of $\bar{F}=F^M\rest\nu$,
 either:
 \begin{enumerate}
  \item $\bar{F}$ is non-type Z and either:
  \begin{enumerate}
  \item $\bar{F}\in\es^M$, or
  \item $\nu$ is a limit of generators of $F^M$
  and $M|\nu$ is active and
  \[ \bar{F}\in\es(\Ult(M|\nu,F^{M|\nu})),\]
  \end{enumerate}
  or
  \item $\bar{F}$ is type Z, $\nu=\mu+1$,
  $\mu$ is a cardinal of $M$, and there is a normal measure 
$G\in\es(\Ult(M,F^M\rest\mu))$
  with $\crit(G)=\mu$, with
  $\bar{F}=G\com(F^M\rest\mu)$.\qedhere
 \end{enumerate}

\end{dfn}

\begin{dfn}
 Let $\mathscr{E}$ be some class of extenders.
 The \emph{J-$\mathscr{E}$-good maximal $L[\es]$-construction of length 
$\lambda$}
 is the unique sequence $\left<N_\alpha\right>_{\alpha<\lambda}$ of J-pms 
$N_\alpha$
 such that:
 \begin{enumerate}
  \item $N_0=V_\om$,
  \item for limit $\eta<\lambda$, $N_\eta=\liminf_{\alpha<\eta}N_\alpha$,
  \item given $\alpha+1<\lambda$, if $N_\alpha$ has largest cardinal $\theta$
  and there is $N$ such that:
  \begin{enumerate}
   \item\label{item:N_is_J-pm} $N$ is an active J-pm satisfying the MS-ISC,
   \item either $N^\passive=N_\alpha$, or $N_\alpha\pins N$ with
$\OR^{N_\alpha}=(\theta^+)^N$,
   \item $\theta\leq\nu(F^N)<\OR^{N_\alpha}$,
   \item\label{item:F^N_is_bkgdd} there is $E\in\mathscr{E}$ such that 
$F^N\rest\nu(F^N)\sub E$
  and $\mathrm{strength}(E)=\nu(F^N)$,
  \end{enumerate}
then letting $\nu=\nu(F^N)$ be least among all such $N$,
there is a unique such $N$ with 
$\nu(F^N)=\nu$, and this is $N_{\alpha+1}$;
 otherwise there is no $N$ satisfying conditions 
\ref{item:N_is_J-pm}--\ref{item:F^N_is_bkgdd},
and $N_\alpha$ is $\om$-solid and $\om$-universal and
 $N_{\alpha+1}=\J(\core_\om(N))$.
 \end{enumerate}

 So if we reach $\alpha$ such that $N_\alpha$ is non-$\om$-solid
 or non-$\om$-universal,
 or there are two distinct J-pms $N,N'$
 with $\nu(F^N)=\nu(F^{N'})$ minimal for satisfying conditions 
\ref{item:N_is_J-pm}--\ref{item:F^N_is_bkgdd},
 then $\lambda=\alpha+1$. 
\end{dfn}

\begin{lem}\label{lem:card_proj_stages}
 Let $\left<N_\alpha\right>_{\alpha<\lambda}$
 be a J-$\mathscr{E}$-good maximal $L[\es]$-construction.
 Let $\alpha<\lambda$.
 Then:
 \begin{enumerate}
 \item\label{item:proj_stage} Let $M\pins N=N_\alpha$ be such that $\rho_\om^M$ 
is an $N$-cardinal.
 Suppose that if $N$ is active then $\rho_\om^M\leq\nu(F^N)$.
 Then there is a unique $\beta<\lambda$ such that $\core_\om(N_\beta)=M$,
 and moreover, $\beta<\alpha$.
 
\item\label{item:card_stability} If $\beta<\alpha$
 and $\rho_\om^{N_\beta}\leq\rho_\om^{N_\gamma}$ for all 
$\gamma\in(\beta,\alpha)$
 then $\core_\om(N_\beta)\pins N_\alpha$ and $\rho_\om^{N_\beta}$
 is an $N_\alpha$-cardinal.
\end{enumerate} 
\end{lem}
\begin{proof}
 By induction on $\alpha$. Part \ref{item:card_stability} is proved as usual, so 
we omit it.
 
 Suppose $\alpha=\beta+1$.

 If $N_\alpha$ is active,
 just apply induction to $N_\beta$; this works as
 $\theta\leq\nu(F^{N_\alpha})$ where $\theta$ is the largest cardinal of 
$N_\beta$
 (this is the key difference to the standard version of this 
lemma).
 Otherwise $N_\alpha=\J(\core_\om(N_\beta))$.
 So if $M\pins N_\alpha$ and $\rho=\rho_\om^M$ is an $N_\alpha$-cardinal
 then $\rho\leq\rho_\om^{N_\beta}$.
 Use induction and universality of the standard parameter basically as usual,
 noting that if $N_\beta$ is active then 
$\rho_\om^{N_\beta}\leq\nu(F^{N_\beta})$.
 
 The limit case is as usual. 
\end{proof}

\begin{dfn}
A \emph{reasonable structure}
is a transitive structure $R=(S,\mathscr{E})$
such that $S$ is transitive, $R\sats\ZFC$,
$\mathscr{E}\sub S$ is a class of $S$-extenders,
and for each $E\in\mathscr{E}$ and $R$-cardinal $\lambda$,
we have $E\rest\lambda\in\mathscr{E}$.
\footnote{$\ZFC$ itself is not particularly important;
it will be clear that we could make do with much less.};

Definability over reasonable structures $R=(S,\mathscr{E})$
is with respect to the predicate $\mathscr{E}$, including the fact that 
$R\sats\ZFC$.
And iterability for such $R$
is for iteration trees which only use extenders
from $\mathscr{E}$ and its images.
\end{dfn}

\begin{dfn}
Let $R=(S,\mathscr{E})$ be an $(\om_1+1)$-iterable  reasonable structure.

Let $\left<N_\alpha\right>_{\alpha<\lambda}$ 
be a J-$\mathscr{E}$-good maximal $L[\es]$-construction of $R$
and $\alpha<\lambda$ and $m<\om$ be such that
$N_\alpha$ is $m$-solid, and suppose that all proper segments
of $N_\alpha$ satisfy standard condensation facts. Let 
$N=\core_m(N_\alpha)$.
Let $\pi:M\to N$ be an $m$-lifting embedding.
Let $E\in\es_+^M$ be such that the reverse-e-dropdown of $(M,E)$
is just $\left<E\right>$.
Let $\left<M_i\right>_{i\leq k}$ be the reverse-mod-dropdown of $(M,M|\lh(E))$.
We define the \emph{$(\alpha,m,\pi,E)$-resurrection}
\[ \left<\alpha_j,m_j,\pi_j,\sigma_j\right>_{i\leq k} \]
of $R$, with:
\begin{enumerate}[label=--]
 \item  $\alpha_j\leq\alpha$ and $m_j\leq\om$,
 \item $\pi_j:M_j\to\core_{m_j}(N_{\alpha_j}^R)$
is an  $m_j$-lifting embedding,
\item $(\alpha_0,m_0,\pi_0)=(\alpha,m,\pi)$,
\item if $j>0$ and $m_j<\om$ then $\rho_{m_j+1}^{M_j}=\rho_\om^{M_j}$,
\item  $\sigma_j=\tau^R_{\alpha_jm_j0}\com\pi_j:M_j\to N_{\alpha_j}^R$,
where
$\tau^R_{\beta k\ell}:\core_k(N_\beta^R)\to\core_\ell(N_\beta^R)$
is the core embedding,
\item if $j<k$ then $\sigma_j\rest\rho_\om^{M_{j+1}}\sub\pi_{j+1}$.
\end{enumerate}

Suppose we have defined $\alpha_j,m_j,\pi_j,\sigma_j$, 
and 
$k>j$.
We have
$\sigma_j:M_j\to N_{\alpha_j}^R$.
Let $\rho=\rho_\om^{M_{j+1}}$. So $\rho$ is an $M_j$-cardinal,
and if $M_j$ is active then $\rho\leq\nu(F^{M_j})$.

If $M_j$ is passive or $\sigma_j(\rho)\leq\nu(F(N_{\alpha_j}^R))$
then by \ref{lem:card_proj_stages} we can set $\alpha_{j+1}$
to be the unique $\alpha'$ such that
$\sigma_j(M_{j+1})=\core_\om(N_{\alpha'}^R)$,
and set $\pi_{j+1}=\sigma_j\rest M_{j+1}$
and $m_{j+1}=\om$.

Suppose $M_j$ is active and
$\sigma_j(\rho)>\nu=\nu(F(N_{\alpha_j}^R))$.
Let $m<\om$ be least such that $\rho_{m+1}^{M_{j+1}}=\rho$.
Let $M^+=\sigma_j(M_{j+1})$,
\[ \bar{M}=\cHull_{m+1}^{M^+}(\nu\cup\{\pvec_{m+1}^{M^+}\}) \]
and $\bar{\pi}:\bar{M}\to M^+$ be the uncollapse,
noting that $\bar{M}$ is $(m+1)$-sound
with $\nu=\rho_{m+1}^{\bar{M}}$,
because $\bar{M}\in N^R_{\alpha_j}$,
where $\nu$ is a cardinal,
and $\sigma_j``M_{j+1}\sub\rg(\bar{\pi})$.
By condensation, $\bar{M}\pins N^R_{\alpha_j}$,
and $\rho_\om^{\bar{M}}=\nu$. Therefore we can set $\alpha_{j+1}$ to be
the unique $\alpha'$ such that
$\bar{M}=\core_\om(N_{\alpha'}^R)$,
and set  $m_{j+1}=m$ and
\[ \pi_{j+1}=\bar{\pi}^{-1}\com\sigma_j\rest M_{j+1}.\qedhere\]
\end{dfn}

\begin{tm} Let $R=(S,\mathscr{E})$ be an $(\om_1+1)$-iterable  reasonable 
structure.

Then $R$ has a J-$\mathscr{E}$-good maximal $L[\es]$-construction 
$\left<N_\alpha\right>_{\alpha<\lambda}$ of length $\lambda=\OR^R+1$,
and for every $\alpha<\lambda$ and $m<\om$,
$\core_m(N_\alpha)$ is J-$(m,\om_1,\om_1+1)^*$-iterable. The final model 
$N_\lambda$ is 
a proper class of $R$ which models $\ZFC$.

Moreover, if $\delta\in\OR^R$ and $R\sats$``$\delta$ is Woodin as witnessed by 
extenders in $\mathscr{E}$'',
then $N_\lambda\sats$``$\delta$ is Woodin''.
\end{tm}

\begin{proof}[Proof sketch]
The overall proof is mostly standard, with just some little
differences. We describe enough to mention these
and give the main structure; the remaining details
will be very routine.

 We first consider the iterability.
By Theorem \ref{tm:lifting-iter_equiv} it suffices to see that 
$M=\core_m(N_\alpha)$ is 
lifting-MS-$(m,\om_1+1)$-iterable.
Fix an $(\om_1+1)$-strategy $\Gamma$ for $R$ (with respect to the class 
$\mathscr{E}$).
We will define a lifting-MS-$(m,\om_1+1)$ strategy
$\Sigma$ for $M$, lifting to trees on $R$ via $\Gamma$.
We mostly keep track of the usual kind of data and maintain the usual kinds of 
inductive properties.
But some details, particularly to do with resurrection, are a little different 
than usual. Write $\CC=\left<N^R_\alpha\right>_{\alpha<\lambda}$,
where $\lambda$ is as large as possible that this is defined.
Fix $\xi<\lambda$, $m\leq\om$ such that $N^R_\xi$
is $m$-solid,
an $m$-sound premouse $M$ and an $m$-lifting embedding
$\pi:M\to\core_m(N^R_\xi)$.

We will define trees $\Tt,\Uu$ on $M,R$ respectively,
with $\Tt$ being lifting-MS-$m$-maximal
and $\Uu$ a coarse tree, via $\Gamma$,
with models $M_\alpha=M^\Tt_\alpha$
and $R_\alpha=R^\Uu_\alpha$,
degrees $d_\alpha=\deg^\Tt_\alpha$,
and embeddings $i_{\alpha\beta}=i^\Tt_{\alpha\beta}$,
$i^*_\alpha=i^{*\Tt}_\alpha$,
$j_{\alpha\beta}=i^\Uu_{\alpha\beta}$,
$j^*_\alpha=i^{*\Uu}_\alpha$, and
for $\alpha<\lh(\Tt)$, the objects
 $\CC_\alpha,\xi_\alpha,d_\alpha,\pi_\alpha,
 \xi^*_\alpha,\pi^*_\alpha,\sigma_\alpha$,
such that:
\begin{enumerate}
 \item ${<_\Tt}={<_\Uu}$,
 \item $\CC_\alpha=i^\Uu_{0\alpha}(\CC)$ and $\xi_\alpha<\lh(\CC_\alpha)$,
 \item   $(\xi_0,\pi_0)=(\xi,\pi)$ (and $d_0=m$),
 \item $\pi_\alpha:M_\alpha\to \core_{d_\alpha}(N_{\xi_\alpha}^{R_\alpha})$ is a
 $d_\alpha$-lifting embedding,
 \item If $\alpha$ is  a successor 
 and $\beta=\pred^\Tt(\alpha)$ and $\alpha\notin\dropset^\Tt$
 then $\xi^*_\alpha=\xi_\beta$
 and
 \[ \pi^*_\alpha=\tau^{R_\beta}_{\xi_\beta d_\beta d_\alpha}:
  M^{*\Tt}_\alpha\to\core_{d_\alpha}(N^{R_\beta}_{\xi^*_\alpha}).
 \] 
\item If $\alpha$ is a successor and $\beta=\pred^\Tt(\alpha)$
and $\alpha\in\dropset^\Tt$, then letting
\[ (\widetilde{\alpha}_i,\widetilde{m}_i,\widetilde{\pi}_i,
\widetilde{\sigma}_i)_{i\leq k}
\]
be the $(\xi_\beta,d_\beta,\pi_\beta,E^\Tt_\beta)$-resurrection
of $R_\beta$, and $\left<\widetilde{M}_i\right>_{i\leq k}$
be the rev-mod-dropdown of $(M^\Tt_\beta,E^\Tt_\beta)$,
and $i$ such that $M^{*\Tt}_\alpha=\widetilde{M}_i$,
then $\xi^*_\alpha=\widetilde{\alpha}_i$ and
\[ \pi^*_\alpha=\tau^{R_\beta}_{\xi^*_\alpha\widetilde{m}_i d_\alpha}
 \com\widetilde{\pi}_i: M^{*\Tt}_\alpha\to\core_{d_\alpha}
 (N^{R_\beta}_{\xi^*_\alpha}),
\]
noting that $d_\alpha\leq\widetilde{m}_i$
since $\rho_{\widetilde{m}_i+1}(M^{*\Tt}_\alpha)\leq\crit(E^\Tt_\alpha)<
\rho_{\widetilde{m}_i}(M^{*\Tt}_\alpha)$.

\item If $\alpha+1<\lh(\Tt)$ then letting 
$(\widetilde{\alpha}_i,\widetilde{m}_i,\widetilde{\pi}_i,
 \widetilde{\sigma}_i)_{i\leq k}$
be the $(\xi_\alpha,d_\alpha,\pi_\alpha,E^\Tt_\alpha)$-resurrection
of $R_\alpha$, then
\[ \sigma_\alpha=\widetilde{\sigma}_k:\exit^\Tt_\alpha\to 
N^{R_\alpha}_{\widetilde{\alpha}_k},\]
a $0$-lifting embedding,
and $E^\Uu_\alpha\in\mathscr{E}^{R_\alpha}$ is a background
for $F^N$ where $N=N^{R_\alpha}_{\widetilde{\alpha}_k}$,
and $R_\alpha\sats$``$E^\Uu_\alpha$ has
strength exactly $\nu(F^N)$''.
\item If $\alpha+1<\lh(\Tt)$
and $\beta=\pred^\Tt(\alpha+1)$
then $\xi_{\alpha+1}=i^\Uu_{\beta,\alpha+1}(\xi^*_{\alpha+1})$
and
\[ 
\pi_{\alpha+1}
\Big(\Big[a,f^{M^*}_{\tau q}\Big]^{M^{*}}_
{E^\Tt_\alpha}\Big)=
\Big[\sigma_\alpha(a),f^{M'}_{\tau
q'}\Big]^{R_\beta}_{E^\Uu_\alpha},\]
where $M'=\core_{d_{\alpha+1}}(N^{R_\beta}_{\xi^*_{\alpha+1}})$ and 
$q'=\pi^*_{\alpha+1}(q)$,
$\tau$ is an $\rSigma_{d_{\alpha+1}}$ term,
and for $N$ a $d_{\alpha+1}$-sound premouse
and $q\in N$, $f_{\tau q}^N$ is the function
$x\mapsto\tau^N(q,x)$.

Here letting $\kappa=\crit(E^\Tt_\alpha)$,
we have
$\sigma_\alpha\rest\pow(\kappa)\inter\exit^\Tt_\alpha=
 \pi^*_\alpha\rest\pow(\kappa)\inter M^{*\Tt}_\alpha$.
 \item If $\alpha<\beta<\lh(\Tt)$ and $\nu=\nu(E^\Tt_\alpha)$
 then
\[ 
\sigma_\alpha\rest(M^\Tt_\alpha|\nu)\sub\pi_{\beta}
\text{ and }\sigma_\alpha(\nu)\leq\pi_\beta(\nu),\sigma_\beta(\nu)\]
and if $\nu$ is not a cardinal of $\exit^\Tt_\alpha$
and $\iota=(\nu^+)^{\exit^\Tt_\alpha}$ 
then
\[ \sigma_\alpha\rest(M_\alpha|\iota)\sub\pi_{\beta}\text{
and }\sigma_\alpha(\iota)\leq\pi_\beta(\iota),\sigma_\beta(\iota).\]
 
\item 
If $\alpha$ is  a successor then $\pi_\alpha\com 
i^{*\Tt}_\alpha=i^{*\Uu}_\alpha\com\pi^*_\alpha$.
\item If $\alpha<^\Uu\beta$ and 
$(\alpha,\beta]_\Tt\inter\dropset^\Tt=\emptyset$
then $i^\Uu_{\beta\alpha}
(\xi_\beta)=\xi_\alpha$
 and
 \[ \pi_\alpha\com 
i^\Tt_{\beta\alpha}=
i^\Uu_{\beta\alpha}\com
\tau^{R_\beta}_{\xi_\beta d_\beta d_\alpha}\com \pi_\beta=
\tau^{R_\alpha}_{\xi_\alpha d_\beta d_\alpha}\com
i^\Uu_{\beta\alpha}\com\pi_\beta.\]
\end{enumerate}
This determines the entire process, and the propagation
of the inductive hypotheses is quite routine,
so we leave it to the reader.
(As usual, one maintains in fact a little more agreement
between maps $\pi_\alpha,\pi_\beta$, etc,
than that stated above.) There is a very similar
construction
done in detail in \cite[\S5]{premouse_inheriting}.

But let us remark that a key point is that the resurrections referred to above 
are well-defined,
because the rules of MS-lifting-$m$-maximal 
trees ensure that the rev-e-dropdown
of $(M^\Tt_\beta,E^\Tt_\beta)$
is just $\left<E^\Tt_\beta\right>$. (That is,
 suppose
we have produced $\Tt,\Uu$ through length $\beta+1$,
and $E\in\es_+(M^\Tt_\beta)$, and
 $\lh(E^\Tt_\alpha)<\lh(E)$ for all \emph{stable} $\alpha$,
 that is, all $\alpha$
such that $\lh(E^\Tt_\alpha)\leq\lh(E^\Tt_\gamma)$
for all $\gamma\geq\alpha$.
Let $\left<G_i\right>_{i\leq k}$
be the rev-e-dropdown of $(M^\Tt_\beta,E)$, so $G_k=E$.
Then when we set $E^\Tt_{\beta+i}=G_i$ for $0\leq i\leq k$,
note that the rev-e-dropdown of $(M^\Tt_{\beta+i},G_i)$
is just $\left<G_i\right>$.)

(Also, we are assuming that $R$ is $(\om_1+1)$-iterable,
meaning without restriction on the form of the trees.
But for our purposes here, it actually suffices to assume
that $R$ is $(\om_1+1)$-iterable for stacks of normal trees.
For  lifting-MS-$m$-maximal trees on $M$
lift to normal trees on $R$ by the conditions specified above,
and then this extends to stacks as usual;
this uses Lemma \ref{lem:nu_inc}.)

This gives that the construction does not break down
due to cores not existing, but we also need to see
it does not break down due to non-uniqueness of next extenders.
For this, use a typical bicephalus comparison.
This is basically like in \cite{fsit}, but slightly different,
with features as in the bicephalus arguments
in \cite{premouse_inheriting}, which the reader should
consult for more details. We just give a sketch here.
Suppose we reach a passive model $N_\alpha$,
and there are two plausible active backgrounded extensions $P,Q$ of $N_\alpha$ 
(in the sense of the construction) with $\nu=\nu(F^P)=\nu(F^Q)$.
Then let $B=(P,Q,\nu)$,
and compare $B$ with itself, like in a standard bicephalus
comparison. (But note we only know that $P,Q$ agree (strictly) below
$(\nu^+)^P=(\nu^+)^Q=\OR^{N_\alpha}$, and possibly $\OR^{N_\alpha}<\OR^P$ and/or
$\OR^{N_\alpha}<\OR^Q$.) We get the lifting-MS-$0$-maximal
iterability of $B$ just like in the proof above. Here
if $B'=(P',Q',\nu')$ is a non-dropping iterate of $B$,
write $\nu'=\sup i``\nu=\sup j``\nu$ where $i:P\to P'$ and $j:Q\to Q'$
are  the iteration 
maps, so $\nu'=\nu(F^{P'})=\nu(F^{Q'})$
and $P'||((\nu')^+)^{P'}=Q'||((\nu')^+)^{Q'}$.
If we want to use $E\in\es_+^{P'}$ with 
$\OR^{N_\alpha}\leq\lh(E)$,
then we use the rev-e-dropdown
of $(P',E)$ to determine the next few extenders. Likewise
for $E\in\es_+^{Q'}$.\footnote{Further
related bicephalus arguments can be seen in \cite{premouse_inheriting}.} We 
don't need
to convert this to J-$m$-maximal iterability,
because it is straightforward to see that lifting-MS-$m$-maximal
iterability is enough for the comparison argument.
Note that if we are at a stage $\alpha$ (such as $\alpha=0$)
of the comparison, with trees $\Tt,\Uu$,
and $M^\Tt_\alpha=B'=(P',Q',\nu')=M^\Uu_\alpha$ is a common bicephalus,
then $P'\neq Q'$, $((\nu')^+)^{P'}=((\nu')^+)^{Q'}$
and $P',Q'$ project to $\nu'$, so $P'\nins Q'\nins P'$.
So there is a least difference $E,F$
with say $E\in\es_+^{P'}$ and/or $F\in\es_+^{Q'}$.
We then want to use $E,F$, but this is preceded
by the rev-e-dropdown of $(P',E)$ and that of $(Q',F)$.
Note here that $F^{P'}$ and $F^{Q'}$ are included in these
(since if $\nu'<\lgcd(P')$ then $((\nu')^+)^{P'}<\lh(E)$ etc).
Now a little further consideration shows that  there are no
$\alpha,\beta$ such that $E^\Tt_\alpha\neq\emptyset\neq E^\Uu_\beta$ and 
$E^\Tt_\alpha,E^\Uu_\beta$ are compatible,
and this leads to contradiction as usual.

We also need the fact that iterable pseudo-J-premice
satisfy the MS-ISC. Here
 a structure $P=(N,F)$
is a \emph{pseudo-J-premouse} iff $N$ is a passive
J-premouse, $F$ an extender over $N$,
and $P$ satisfies the conditions of being a J-premouse,
and letting $\nu=\nu(F)$ and $\delta=\card^N(\nu)$,
then $F\rest\delta\in\es^P_+$.  Suppose $P$ is such
and is lifting-MS-$(0,\om_1,\om_1+1)^*$-iterable,
but $F$ is not type Z.
Then $P$ satisfies the MS-ISC;
the proof of this is almost identical to that in
\cite[\S10]{fsit}. (Actually lifting-MS-$(0,\om_1+1)$-iterability 
is enough, by essentially
the argument in \cite[Theorem 9.4]{fsfni}.)

The fact that Woodins of $R$ are Woodin
in $N=N^R_{\OR^R}$ now follows by the argument in \cite[\S11]{fsit}. The 
pseudo-J-premice which come up are sufficiently iterable
by the proof above.

This completes the proof.
\end{proof}

\section{Tree conversion for MS-mice}\label{sec:MS_translation}

In this section we detail a conversion procedure
for trees on MS-premice, which is very similar to
that for J-premice.
It is used as a black box in \cite[2.12, 2.14***]{iter_for_stacks},
so this section fills in the missing details from there.
We formally assume the reader is familiar
with \S\ref{sec:J_conversion},
For the most part
we  actually give a complete account, independent
of \S\ref{sec:J_conversion};
this just excludes the 
version of Lemma \ref{lem:conversions_cover_J},
which we leave to the reader (and given what we do describe,
it is an easy exercise to fill that in).

We assume the reader is familiar with the definitions
and basic facts in \cite[\S 
\S2***]{iter_for_stacks}. We aim to prove \cite[Lemmas 2.12, 
2.14***]{iter_for_stacks}.

\begin{dfn}
We say
 $(M,k)$ 
is \emph{suitable} iff either:
\begin{enumerate}[label=--]
 \item $M$ is an MS-premouse of type $\leq 2$
 and $M$ is $k$-sound,
or
\item $M$ is a type 3 MS-premouse and $k\geq 1$
and $M$ is $(k-1)$-sound,
or
\item $k=1$ and $M$ is a $\udash 1$-sound internally MS-indexed active seg-pm
(note then letting
$\nu=\nu(F^M)$, $M^\pm$ is the premouse
with the trivial completion of $F^M\rest\nu$
active 
and $M||(\nu^+)^M=(M^\pm)^\passive$).
\end{enumerate}

Let $(M,k)$ be suitable. Let $\Tt$
be a  padded $\udash k$-maximal tree on $M$.
Given $\alpha<\lh(\Tt)$, say  $\alpha$ is:
\begin{enumerate}[label=--]
\item \emph{$\Tt$-special} iff
$M^\Tt_\alpha$ is a non-premouse
and $\udeg^\Tt_\alpha=0$ (equivalently,
$M^\Tt_\alpha$ fails the MS-ISC),
\item \emph{$\Tt$-very-special ($\Tt$-vs)} iff
$\Tt$-special and $E^\Tt_\alpha=F(M^\Tt_\alpha)$.
\item \emph{$\Tt$-pre-special ($\Tt$-ps)} iff
$M^\Tt_\alpha$ is a non-premouse and
$\udeg^\Tt_\alpha=1$,
\item \emph{$\Tt$-very-pre-special ($\Tt$-vps)}
iff $\Tt$-pre-special and $E^\Tt_\alpha=F(M^\Tt_\alpha)$.
\end{enumerate}
Say $\Tt$ is \emph{nicely padded}
iff:
 \begin{enumerate}
\item If $\alpha$ is $\Tt$-pre-special
and not of form
  $\alpha=\beta+1$ with $E^\Tt_\beta=\emptyset$,
then either $E^\Tt_\alpha=\emptyset$
or $\lh(E^\Tt_\alpha)<\OR((M^\Tt_\alpha)^\pm)$.
\item If $\alpha+1<\lh(\Tt)$ and $E^\Tt_\alpha=\emptyset$
then $\alpha$ is  $\Tt$-pre-special
and not of form $\alpha=\beta+1$ with $E^\Tt_\beta=\emptyset$,
and moreover, $\pred^\Tt(\alpha+1)=\alpha$ and $M^\Tt_{\alpha+1}=M^\Tt_\alpha$
and $\deg^\Tt(\alpha+1)=1$ 
and $\alpha+2<\lh(\Tt)$ and
$\OR((M^\Tt_\alpha)^\pm)<\lh(E^\Tt_{\alpha+1})$.
\item If $E^\Tt_\alpha\neq\emptyset$ then define
$\widetilde{\nu}^\Tt_\alpha=\iota(\exit^\Tt_\alpha)$,
and if $E^\Tt_\alpha=\emptyset$ then define 
$\widetilde{\nu}^\Tt_\alpha=\nu(F(M^\Tt_\alpha))$.
We use $\widetilde{\nu}^\Tt_\alpha$ as the exchange ordinal
associated to $\alpha$ in $\Tt$; that is,
if $E^\Tt_\beta\neq\emptyset$ then $\pred^\Tt(\beta+1)$
is the least $\alpha$ such that $\crit(E^\Tt_\beta)<\widetilde{\nu}^\Tt_\alpha$.
\end{enumerate}
If $\alpha+1<\lh(\Tt)$, we say that $\alpha$ is a  \emph{transition point}
of $\Tt$ iff $E^\Tt_\alpha=\emptyset$.

We say that $(M,k,\Tt)$ is \emph{suitable}
iff $(M,k)$ is suitable and $\Tt$ is a nicely padded
$\udash k$-maximal tree on $M$.
\end{dfn}

\begin{lem}\label{lem:core_appears}
 Let $(M,k,\Tt)$ be suitable.
Let $\gamma$ be $\Tt$-special.
Then there is a unique $\beta<\gamma$
such that  $\udeg^\Tt(\beta)=1$
(so $\beta$ is non-$\Tt$-special) and
\[ \core(F(M^\Tt_\gamma))=F((M^\Tt_\beta)^\pm) \]
and $E^\Tt_\beta\neq\emptyset$. Moreover, letting $F=\core(F(M^\Tt_\gamma))$
and $\nu=\nu_F$:
\begin{enumerate}
 \item $\beta=\alpha+1$ where $\alpha$ is a transition point of $\Tt$
 \item $\widetilde{\nu}^\Tt_\alpha=\nu$ and $\widetilde{\nu}^\Tt_{\alpha'}<\nu$
for all $\alpha'<\alpha$
and $\lh(E^\Tt_{\alpha'})\leq\nu$ for all $\alpha'<\alpha$ such that
$E^\Tt_{\alpha'}\neq\emptyset$,
 \item $\alpha+1<^\Tt\gamma$; let $\xi+1=\successor^\Tt(\alpha+1,\gamma)$,
 \item $\udeg^\Tt_{\xi+1}=\udeg^\Tt_\gamma=0$ and 
$(\alpha+1,\gamma]_\Tt\inter\dropset^\Tt=\emptyset$,
 \item\label{item:ext_assoc} $\nu\leq\crit(E^\Tt_\xi)$,
and $F(M^\Tt_\gamma)$ is equivalent to the composition
of the extenders 
\[ \left<F\right>\conc\left<E^\Tt_\delta\right>_{\delta+1\in[\xi+1,\gamma]
_\Tt}\]
(and for $\delta_0,\delta_1\in[\xi+1,\gamma]_\Tt$
with $\delta_0<\delta_1$, we have 
$\nu(E^\Tt_{\delta_0})\leq\crit(E^\Tt_{\delta_1})$).
\end{enumerate}
\end{lem}

\begin{proof}[Proof sketch] The proof is straightforward.
Part \ref{item:ext_assoc} is as in \cite[Lemma 2.27***]{extmax}
(extended to transfinite iterations in a routine manner).
\end{proof}

\begin{dfn}Let $(M,k,\Tt)$ be suitable
and $\gamma+1<\lh(\Tt)$.
If $\gamma$ is $\Tt$-vs,
write $\widetilde{\gamma}$
for the unique transition point $\alpha$ as in Lemma \ref{lem:core_appears}.
If $\gamma$ is $\Tt$-vps, write $\widetilde{\gamma}=\alpha$
where $\alpha+1=\gamma$ (so again, $\alpha$ is a transition point).
Otherwise write $\widetilde{\gamma}=\gamma$. 
\end{dfn}

\begin{dfn}
 Let $(M,k,\Tt)$ be suitable.
 Write  $\beta\leq^{\ext}_{\direct}\alpha$
 iff $\beta=\alpha$ or [$\alpha$ is $\Tt$-vs
 and $\beta+1\in(\widetilde{\alpha}+1,\alpha]_\Tt$].
 Let $\leq^{\ext}$ be the transitive closure
 of $\leq^{\ext}_{\direct}$.
 Let $<^{\ext}_{\direct}$ and $<^{\ext}$
 be the strict parts. Note that 
$\alpha<^{\ext}_{\direct}\beta$ implies $\alpha<\beta$,
so $<^{\ext}$ is wellfounded. Write ${\leq^{\ext,\Tt}}={\leq^{\ext}}$, etc.
\end{dfn}

 Note here that if $\beta_0<^{\ext}\beta_1$ then
 $\beta_0<\beta_1$ but 
$\crit(\core(E^\Tt_{\beta_1}))<\crit(\core(E^\Tt_{\beta_0}))$.

\begin{dfn}
 Let $(M,k,\Tt)$ be suitable and $\alpha+1<\lh(\Tt)$.
 Then the \emph{standard decomposition} of $E^\Tt_\alpha$
 is the enumeration 
 of $\{\core(E^\Tt_\beta)\bigm|\beta\leq^{\ext,\Tt}\alpha\}$,
 in order of increasing critical point.
\end{dfn}

\begin{lem}
 Let $(M,k,\Tt)$ be suitable and $\alpha+1<\lh(\Tt)$.
 
 The standard decomposition of $E^\Tt_\alpha$ 
 is well-defined. That is, if $\beta_0,\beta_1\leq^{\ext,\Tt}\alpha$
 and $\beta_0\neq\beta_1$ then $\kappa_{\beta_0}\neq\kappa_{\beta_1}$
 where $\kappa_{\beta_i}=\crit(\core(E^\Tt_{\beta_i}))=\crit(E^\Tt_{\beta_i})$;
 moreover, if $\beta_0=\alpha$ then $\kappa_{\beta_0}<\kappa_{\beta_1}$,
 and if $\kappa_{\beta_0}<\kappa_{\beta_1}$
 then $\nu(\core(E^\Tt_{\beta_0}))\leq\kappa_{\beta_1}$.

 Further,  $E^\Tt_\alpha$ is equivalent to iteration via the
 extenders in the
  standard decomposition
  of
  $E^\Tt_\alpha$ (in order of increasing critical point).
\end{lem}

\begin{proof}
 This is by induction using Lemma \ref{lem:core_appears},
again as in \cite[Lemma 2.27***]{extmax}.
\end{proof}

\begin{dfn}
 Let $(M,k,\Tt)$ be suitable.
 Let $\alpha<\lh(\Tt)$.
Given $\gamma\leq^\Tt\alpha$,
 $\vec{E}^\Tt_{\gamma\alpha}$ denotes the sequence
$\left<E^\Tt_\beta\right>_{\gamma<^\Tt\beta+1\leq^\Tt\alpha}$
(so $\vec{E}^\Tt_{\gamma\alpha}$ corresponds to $i^\Tt_{\gamma\alpha}$
when the latter exists),
and $\vec{F}^\Tt_{\gamma\alpha}$ denotes
the sequence 
$\left<E^\Tt_\beta\right>_{\beta+1\in[\xi+1,\alpha]
_\Tt}$,
where $\xi$ is least such that
$\gamma<^\Tt\xi+1\leq^\Tt\alpha$ and $(\xi+1,\alpha]_\Tt$
does not drop in model.
Write $\vec{E}^\Tt_{<\alpha}=\vec{E}^\Tt_{0\alpha}$
and $\vec{F}^\Tt_{<\alpha}=\vec{F}^\Tt_{0\alpha}$.

Given a sequence $\vec{E}=\left<E_\alpha\right>_{\alpha<\lambda}$
of short extenders, we define
$\left<U_\alpha,k_\alpha\right>_{\alpha\leq\lambda}$,
if possible, by induction on $\lambda$, 
as follows. Set $U_0=M$ and $k_0=k$.
Given $U_\eta$  and $k_\eta\leq\om$
are well-defined and $U_\eta$ is a $k_\eta$-sound
seg-pm
and $\eta<\lambda$,
then: if $\crit(E_\eta)<\OR(U_\eta)$
and there is $N\ins U_\eta$ such that $E_\eta$ measures
exactly $\pow([\kappa]^{<\om})\inter N$,
then letting $N\ins U_\eta$ be the largest such,
and letting $n\leq\om$ be largest such that $(N,n)\ins(U_\eta,k_\eta)$
and $\crit(E_\eta)<\udash\rho_{n}^N$, then $
U_{\eta+1}=\Ult_{\udash n}(N,E_\eta)$ and $k_{\eta+1}=n$.
We say there is a \emph{drop in model} 
at $\eta+1$ iff $N\pins U_\eta$.
Given a limit $\eta$ such that $U_\alpha$ is well-defined
for each $\alpha<\eta$,
then $U_\eta$ is well-defined iff there are only finitely many drops in model
${<\eta}$, and then $U_\eta$ is the resulting direct limit
and $k_\eta=\lim_{\alpha<\eta}k_\alpha$.
We now define $\Ult_{\udash k}(M,\vec{E})=U_\lambda$,
if this is well-defined,
and if so, and there is no drop in model,
we define the \emph{iteration map} $i^{M,\udash k}_{\vec{E}}:M\to U_\lambda$ 
resulting naturally from the ultrapower maps.
Also, if there is no drop in model,
or the only drop in model occurs at $1$,
then define $\bar{i}^{M,\udash k}_{\vec{E}}:N\to U_\lambda$ where $N\ins M$
is as above.

We also make analogous definitions for standard
fine structural ultrapowers (as opposed to $\udash$ultrapowers),
with notation $\Ult_k(M,\vec{E})$ and $i^{M,k}_{\vec{E}}$ and 
$\bar{i}^{M,k}_{\vec{E}}$.
\end{dfn}

\begin{rem}
Let $(M,k,\Tt)$ be suitable and $\beta\leq^\Tt\alpha<\lh(\Tt)$.
Then clearly
 $M^\Tt_\alpha=\Ult_{\udash m}(M^\Tt_\beta,\vec{E}^\Tt_{\beta\alpha})$
 and 
$i^\Tt_{\beta\alpha}=i^{M^\Tt_\beta,m}
_{\vec{E}^\Tt_{\beta\alpha}}$ 
where $m=\udeg^\Tt_\beta$ (with one map defined iff the other is),
and likewise 
$i^{*\Tt}_{\gamma+1,\alpha}=\bar{i}^{M^\Tt_\beta,m}_{\vec{E}^\Tt_{\beta\alpha}}
$ when $\pred^\Tt(\gamma+1)=\beta$.
\end{rem}

\begin{dfn}
 Let $(M,k,\Tt)$ be suitable.
 
 We say that $\Tt$ is \emph{unravelled}
 iff, if $\lh(\Tt)=\alpha+1$ then $\alpha$ is non-$\Tt$-special.
 The \emph{unravelling} $\unrvl(\Tt)$ of $\Tt$
 is the unique unravelled $\udash k$-maximal tree $\Ss$ on $M$,
 if it exists, such that (i) $\Tt\ins\Ss$, (ii) if $\lh(\Tt)$ is a limit then 
$\Tt=\Ss$, and (iii) if $\lh(\Tt)=\alpha+1$ then $\beta$ is $\Ss$-very-special 
for each $\beta+1<\lh(\Ss)$
with $\alpha\leq\beta$. (So
$E^\Ss_\beta=F(M^\Ss_\beta)$ for each $\beta\geq\alpha$,
and $\crit(E^\Ss_{\alpha+i+1})<\crit(E^\Ss_{\alpha+i})$,
so $\lh(\Ss)<\lh(\Tt)+\om$; the existence of $\Ss$ just depends on the 
wellfoundedness of the resulting models.)

 We say that $\Tt$ is \emph{everywhere unravelable}
 iff (i) $\unrvl(\Tt\rest\beta)$ exists for every $\beta\leq\lh(\Tt)$,
  and (ii)
 for each transition point $\alpha$,
letting $\Ww=(\Tt\rest(\alpha+2))\conc 
F(M^\Tt_\alpha)$ (a putative nicely padded $\udash k$-maximal tree),
$\Ww$ has wellfounded models and $\unrvl(\Ww)$ exists.
\end{dfn}

\begin{dfn}
 Let $(M,k)$ be suitable. Let $m=k$ if
  $M$ is type $\leq 2$, and $m=k-1$ otherwise.
  Write $m^M(k)=m$.
  We say that $(M,k,m)$ is \emph{suitable},
  and  say $(M,k,\Tt,m)$ is \emph{suitable}
  iff $(M,k,\Tt)$ and $(M,k,m)$ are suitable.

  Let $\Uu$ be an $m$-maximal
 (not $\udash m$-maximal!)
 tree on $M$. We say that $\Uu$ is \emph{$(M,\uu)$-wellfounded}
 iff for every $\alpha<\lh(\Uu)$,
 $U=\Ult_{\udash k}(M,\vec{E}_{<\alpha})$ is wellfounded.
\end{dfn}

With notation as above,
 it is straightforward to see that if $M^\Tt_\alpha$ is 
 type $\leq 2$
 then $U=M^\Tt_\alpha$,
 and if $M^\Tt_\alpha$ is type 3 then 
$U^\passive\ins\Ult(M^\Tt_\alpha|(\kappa^+)^{M^\Tt_\alpha},F(M^\Tt_\alpha))$
where $\kappa=\crit(F(M^\Tt_\alpha))$.
Therefore if $\Uu$ is via a reasonable $m$-maximal strategy
for $M$, then $\Uu$ is $(M,\uu)$-wellfounded.

\begin{dfn}\label{dfn:conversion}
 Let $(M,k,\Tt,m)$ be suitable with $\Tt$ unravelled and everywhere unravelable.
 We define
 a padded  $m$-maximal
 tree $\Uu=\conv(\Tt)$ on $M$,
 the \emph{$m$-maximal conversion} of $\Uu$,
 with exchange ordinals
 $\widetilde{\nu}^\Uu_\alpha$ for 
$\alpha+1<\lh(\Uu)$, by requiring
 (we
 verify later that this works):
 \begin{enumerate}
  \item $\lh(\Uu)=\lh(\Tt)$.  
\item\label{item:tree_length_exchange_ord_agmt} $\widetilde{\nu}
^\Uu_\alpha=\widetilde{\nu}^\Tt_\alpha$ for 
$\alpha+1<\lh(\Uu)$.
  \item if $\alpha$ is a transition point of $\Tt$
  then $E^\Uu_\alpha=F(M^\Uu_\alpha)$ (which
  is the trivial completion of $F(M^\Tt_\alpha)\rest\nu(F(M^\Tt_\alpha))$).
  \item If $\alpha$ is $\Tt$-vs or $\Tt$-vps then $E^\Uu_\alpha=\emptyset$
  and $\pred^\Uu(\alpha+1)=\alpha$ and 
$M^\Uu_{\alpha+1}=M^\Uu_\alpha$
and $\deg^\Uu_{\alpha+1}=\deg^\Uu_\alpha$.
  \item If $\alpha$ is a non-transition point, non-$\Tt$-vs
  and non-$\Tt$-vps then $E^\Uu_\alpha=E^\Tt_\alpha$.
  \item\label{item:conversion_limits} Let $\eta<\lh(\Tt)$ be a limit.
  Fix $\gamma<^\Tt\eta$
  such that $(\gamma,\eta)_\Tt$ does not drop
  and does not contain the successor of a transition point.
  Let $X$ be the set of all $\beta<\eta$ such that 
$\beta\leq^{\ext,\Tt}\alpha$ for some
$\alpha+1\in(\gamma,\eta)_\Tt$ with $E^\Tt_\alpha\neq\emptyset$.
Then:
\begin{enumerate}
\item $\{\widetilde{\beta}+1\bigm|\beta\in X\}$ is cofinal in 
$\eta$;
 \item 
 $E^\Uu_{\widetilde{\beta}}=\core(E^\Tt_{\beta})$
 for each $\beta\in X$;
\item for all $\beta_0,\beta_1\in X$,
either $\widetilde{\beta_0}+1\leq^\Uu\widetilde{\beta_1}+1$ or
vice versa;
\item  $[0,\eta)_\Uu$ is the
$\leq^\Uu$-downward closure
of $\{\widetilde{\beta}+1\bigm|\beta\in X\}$.\qedhere
\end{enumerate}  
 \end{enumerate}
\end{dfn}

\begin{lem}\label{lem:conversion_exists}
  Let $(M,k,\Tt,m)$ be suitable with $\Tt$ unravelled and everywhere 
unravelable, and $\Tt$ non-trivial. Then:
\begin{enumerate}
 \item $\Uu=\conv(\Tt)$ is a well-defined padded $m$-maximal
 tree on $M$.
\item Suppose $\lh(\Tt)=\alpha+1$
(so either $M^\Tt_\alpha$ is type $\leq 2$ 
or $\udeg^\Tt_\alpha\geq 1$).
Let  $\varepsilon^\Tt+1\leq^\Tt\alpha$ be least such that
$(\varepsilon^\Tt+1,\alpha]_\Tt\inter\dropset^\Tt=\emptyset$,
and $\varepsilon^\Uu$ likewise for $\Uu$. Then (and let $\delta,N^*$ be defined 
by):
\begin{enumerate}
 \item $M^\Uu_\alpha=(M^\Tt_\alpha)^\pm$,
\item $[0,\alpha]_\Tt\inter\dropset^\Tt=\emptyset\iff [0,\alpha]
_\Uu\inter\dropset^\Uu=\emptyset$.
\item $\delta=\pred^\Tt(\varepsilon^\Tt+1)=\pred^\Uu(\varepsilon^\Uu+1)$,
\item  
$(N^*)^\pm=M^{*\Uu}_{\varepsilon^\Uu+1}$
where $N^*=M^{*\Tt}_{\varepsilon^\Tt+1}$,
and note if $[0,\alpha]_\Tt\inter\dropset^\Tt\neq\emptyset$
then $N^*=(N^*)^\pm$,
\item $\vec{F}^\Uu_{<\alpha}=
\{\core(E^\Tt_\beta)\bigm|\exists\gamma\ 
\big[\varepsilon^\Tt+1\leq^\Tt\gamma+1\leq^\Tt\alpha\text{ and } 
\beta\leq^{\exit,\Tt}\gamma\big]\}$,
so $\vec{F}^\Uu_{<\alpha}$ is equivalent to 
$\vec{F}^\Tt_{<\alpha}$,
\item letting $n=\udeg^\Tt(\alpha)$ and 
$n'=\deg^\Uu(\alpha)$, we have
\begin{enumerate}[label=--]
\item $M^\Tt_\alpha=\Ult_{\udash 
n}(N^*,\vec{F}^\Tt_{<\alpha})$,
\item $M^\Uu_\alpha=\Ult_{n'}((N^*)^\pm,\vec{F}^\Uu_{<\alpha})$,
\item $i^{*\Uu}_{\varepsilon^\Uu+1,\alpha}=i^{*\Tt}_{\varepsilon^\Tt+1,\alpha}
\rest\core_0((N^*)^\pm)$.
\end{enumerate}

\end{enumerate}

\item $\Uu$ is $(M,\uu)$-wellfounded, and moreover:
\begin{enumerate}
\item For  $\beta\in[1,\lh(\Tt)]$ not of form $\beta=\xi+2$\footnote{
Note that $\Tt\rest\xi+2$ has last model $M^\Tt_{\xi+1}$,
and so if $\xi$ is a transition point then the last ``extender''
of $\Tt\rest(\xi+2)$ is $E^\Tt_\xi=\emptyset$.}
for a transition point $\xi$,
\[\conv(\unrvl(\Tt\rest\beta))=(\Uu\rest\beta)\conc(\emptyset,\ldots,\emptyset)
\]
(where $(\Uu\rest\beta)\conc(\emptyset,\ldots,\emptyset)$
is an extension of $\Uu\rest\beta$ by just padding;
the extension is finitely long), and
\item For each transition point $\xi$,\footnote{So $E^\Tt_\xi=\emptyset$ and 
$E^\Uu_\xi=F(M^\Uu_\xi)$
and $F(M^\Tt_{\xi})=F(M^\Tt_{\xi+1})$ is equivalent to $F(M^\Uu_\xi)$.}
\[ \conv\Big(\unrvl\Big((\Tt\rest(\xi+2))\conc F(M^\Tt_{\xi+1})\Big)\Big)=
 (\Uu\rest(\xi+2))\conc(\emptyset,\ldots,\emptyset).
\]
(The models witnessing $(M,\uu)$-wellfoundedness
appear as the last models of the unravelled trees mentioned in the two 
clauses above, so they are wellfounded.)
\end{enumerate}
\item\label{item:comparability} Let $\alpha+1<\lh(\Tt)$ with 
$E^\Tt_\alpha\neq\emptyset$ and 
$X=\{\beta\bigm|
\beta\leq^{\ext,\Tt}\alpha\}$.
Then for each $\beta\in X$,
we have $E^\Uu_{\widetilde{\beta}}=\core(E^\Tt_\beta)$,
and for all $\beta_0,\beta_1\in X$, if 
$\crit(E^\Tt_{\beta_0})\leq\crit(E^\Tt_{\beta_1})$
then
$\widetilde{\beta_0}+1\leq^\Uu\widetilde{\beta_1}+1\leq^\Uu\alpha+1$
and
$(\widetilde{\beta_0}+1,\alpha+1]
_\Uu\inter\dropset^\Uu_{\deg}=\emptyset$.
\item\label{item:iterated_comparability} Let $\alpha+1<^\Tt\alpha'+1<\lh(\Tt)$
be such that $E^\Tt_\alpha\neq\emptyset\neq E^\Tt_{\alpha'}$
and $\alpha'+1$ is non-$\Tt$-special
and $(\alpha+1,\alpha'+1]_\Tt$ does not drop in model.
Let  $\beta\leq^{\ext,\Tt}\alpha$
and $\beta'\leq^{\ext,\Tt}\alpha'$.
Then:
\begin{enumerate}[label=--]
\item $\widetilde{\alpha}+1\leq^\Uu\widetilde{\beta}
+1\leq^\Uu\alpha+1\leq^\Uu
\widetilde{\alpha'}+1\leq^\Uu\widetilde{\beta'}+1\leq^\Uu\alpha'+1$,
\item $(\widetilde{\alpha}+1,\alpha'+1]_\Uu\inter\dropset^\Uu=\emptyset$, and
\item if
$(\alpha+1,\alpha'+1]_\Tt\inter\dropset_{\deg}^\Tt=\emptyset$
then 
$(\widetilde{\alpha}+1,\alpha'+1]_\Uu\inter\dropset_{\deg}^\Uu=\emptyset$.
\end{enumerate}
\item\label{item:real_exts_of_Uu} If $\Uu$ has successor length,
then for every $\gamma+1\in b^\Uu$ with 
$E^\Uu_\gamma\neq\emptyset$,
there is $\beta+1<\lh(\Tt)$ with $E^\Tt_\beta\neq\emptyset$
and $\gamma=\widetilde{\beta}$, so $E^\Uu_\gamma=\core(E^\Tt_\beta)$.
\end{enumerate}
\end{lem}
\begin{proof}
The proof is by induction on $\lh(\Tt)$.
For $\lh(\Tt)=1$ it is trivial and for $\lh(\Tt)$ a limit,
it follows immediately by induction. So suppose $\lh(\Tt)=\varepsilon+1$
for some $\varepsilon>0$.

\begin{casethree} $\lh(\Tt)=\xi+n+2$ where  $n<\om$ and
$\xi$ is non-$\Tt$-vs but $\xi+1+i$ is $\Tt$-vs for all $i<n$.

Note then that $E^\Tt_\xi\neq\emptyset$ (according to the rules
of nicely padded trees).

\begin{scasethree}
 It is not the case that $\xi=\varepsilon+1$ for a transition point 
$\varepsilon$.

Let $\mu\leq\xi$ be least such that $\alpha$ is $\Tt$-vs
for each 
$\alpha\in[\mu,\xi)$. Let
\[ \bar{\Tt}=\unrvl(\Tt\rest(\mu+1)),\]
and say $\lh(\bar{\Tt})=\mu+\ell+1$ (so $\ell<\om$).
Let $\bar{\Uu}=\conv(\bar{\Tt})$.
So $\xi\in[\mu,\mu+\ell]$ 
and by induction, we have
\[ 
(M^{\bar{\Tt}}_{\mu+\ell})^\pm=M^{\bar{\Uu}}_{\mu+\ell}=M^{\bar{\Uu}}
_\mu=M^\Uu_\mu=M^\Uu_\xi.\]
Since $\xi$ is non-$\Tt$-vs
and by  subcase hypothesis,
therefore $E^\Tt_\xi\in\es_+(M^\Uu_\xi)$,
and note that $\lh(E^\Uu_\alpha)\leq\lh(E^\Tt_\xi)$
for each $\alpha<\xi$ with $E^\Uu_\alpha\neq\emptyset$. So we can set 
$E^\Uu_\xi=E^\Tt_\xi=E$.
Let $\kappa=\crit(E)$.
Let $\chi=\pred^\Tt(\xi+1)=\pred^\Uu(\xi+1)$ (recall 
$\widetilde{\nu}^{\bar{\Tt}}_\beta=\widetilde{\nu}^{\bar{\Uu}}_\beta$ for all 
$\beta+1<\lh(\bar{\Tt})$,
by \ref{dfn:conversion}).

\begin{sscasethree} $\chi$ is non-$\Tt$-special and not the successor of a 
transition point.
 
 So $(M^\Tt_\chi)^\pm=M^\Uu_\chi$
 and  
$\kappa<\widetilde{\nu}^\Tt_\chi=\widetilde{\nu}
^\Uu_\chi=\nu(E^\Uu_\chi)$, and either $\chi$ is a transition point
and $E^\Tt_\chi=\emptyset$ and $E^\Uu_\chi=F(M^\Uu_\chi)$
and $M^\Uu_\chi$ is active type 3,
or $\chi$ is a non-transition point and $E^\Tt_\chi=E^\Uu_\chi$.  If 
$M^\Tt_\chi\neq M^\Uu_\chi$
then $(M^\Uu_\chi)^\passive=M^\Tt_\chi||\OR(M^\Uu_\chi)$
and $\OR(M^\Uu_\chi)$ is a cardinal of $M^\Tt_\chi$,
and therefore $\xi+1\in\dropset^\Tt$ iff $\xi+1\in\dropset^\Uu$,
and $\xi+1\in\dropset^\Tt_{\udeg}$ iff $\xi+1\in\dropset_{\deg}^\Uu$,
and when there is a drop, the drops are to the same model and
corresponding $\udash$degree  and degree respectively.
Note 
that $\xi+1$ is non-$\Tt$-special, so $E$ is the last extender
used in $\Tt,\Uu$, and $\lh(\Tt)=\xi+2$.

We claim that $(M^\Tt_{\xi+1})^\pm=M^\Uu_{\xi+1}$, and 
there is appropriate
agreement of iteration maps. This is immediate
when $M^{*\Uu}_{\xi+1}$ is non-type 3,
so suppose it is type 3.
Suppose first there is no drop in model at $\xi+1$.
So possibly
 $M^\Tt_\chi\neq M^\Uu_\chi$, and in any case,
 letting $d=\udeg^\Tt_{\xi+1}$ and $e=\deg^\Uu_{\xi+1}$
 (so either $d=e+1<\om$ or $d=e=\om$),
then $M^\Tt_{\xi+1}=\Ult_{\udash d}(M^\Tt_\chi,E)$
(formed without squashing), whereas
$M^\Uu_{\xi+1}=\Ult_{e}(M^\Uu_\chi,E)$ (formed with squashing).
By \cite[Definition 2.5]{iter_for_stacks} and
as in \cite[Lemma 9.1]{fsit}, we get
$(M^\Tt_{\xi+1})^\pm=M^\Uu_{\xi+1}$
and the ultrapower maps agree over $(M^\Uu_\chi)^\sq$.
When there is a drop in model, it is likewise,
but slightly simpler, because then
we have $M^{*\Tt}_{\xi+1}=M^{*\Uu}_{\xi+1}$.

The remaining properties in this subsubcase are now straightforward
to verify by induction.
\end{sscasethree}

\begin{sscasethree}\label{sscase:chi_non_T-special_is_trans_succ} 
$\chi=\alpha+1$
for a
transition point $\alpha$ (so $\chi$ is non-$\Tt$-special).

By subcase hypothesis, $\chi<\xi$.
So $M^\Uu_\alpha$ is type 3, $\udeg^\Tt(\alpha)\geq 1$ and
 $(M^\Tt_\alpha)^\pm=M^\Uu_\alpha$
 and 
$E^\Uu_\alpha\rest\nu=
F(M^\Tt_\alpha)\rest\nu$ where 
$\nu=\widetilde{\nu}^{\Tt,\Uu}_\alpha=\nu(E^\Uu_\alpha)=\nu(F(M^\Tt_\alpha))$.
 With $\theta=\crit(F(M^\Tt_\alpha))=\crit(E^\Uu_\alpha)$, note 
\[\begin{array}{rcl}
(M^\Tt_\alpha)^\passive&=&
\Ult(M^\Tt_\alpha|(\theta^+)^{M^\Tt_\alpha},
F(M^\Tt_\alpha))|\OR^{M^\Tt_\alpha}\\
&=&\Ult(M^{*\Uu}_{\alpha+1}|(\theta^+)^{M^{*\Uu}_{\alpha+1}},
E^\Uu_\alpha)|\OR^{M^\Tt_\alpha}\\
&=&M^\Uu_{\alpha+1}||(\delta^+)^{M^\Uu_{\alpha+1}},\end{array}\]
where $\delta=\lgcd(M^\Tt_\alpha)$. Moreover,
$\delta\neq\lgcd(M^\Uu_{\alpha+1})$,
because otherwise, $\theta=\lgcd(M^{*\Uu}_{\alpha+1})$
and  $M^{*\Uu}_{\alpha+1}$ is active type 2
and $F(M^{*\Uu}_{\alpha+1})=E^\Uu_\beta$ for some $\beta<\alpha$,
but $M^\Uu_\alpha$ is active with $\crit(F(M^\Uu_\alpha))=\theta$,
and it is easy to see this gives a contradiction. So
$(M^\Tt_{\alpha+1})^\passive\pins M^\Uu_{\alpha+1}$.
Since $\alpha+1<\xi$ and by induction,
letting
\[ \Tt'=\unrvl((\Tt\rest(\alpha+2))\conc F(M^\Tt_{\alpha+1})) \]
(note $M^\Tt_{\alpha+1}=M^\Tt_\alpha$),
then $\Tt'$ exists (with wellfounded models)
and 
letting
\[ \Uu'=\conv(\Tt')=\Uu\rest(\alpha+2)\conc(\emptyset,\ldots,\emptyset),\]
then $(M^{\Tt'}_\infty)^\pm=M^{\Uu'}_\infty=M^\Uu_{\alpha+1}$
and letting $e=\deg^\Uu_{\alpha+1}=\deg^{\Uu'}_\infty$
and  $d=\udeg^{\Tt'}(\infty)$, then $e=m^{M^{\Tt'}_\infty}(d)$ (so 
$e=d$
or $e=d-1$ as appropriate).

Now 
$\nu\leq\kappa<\widetilde{\nu}^\Tt_{\alpha+1}=\widetilde{\nu}^\Uu_{\alpha+1}$.
Since $\OR(M^\Tt_{\alpha+1})=\OR(M^\Tt_\alpha)$ is a cardinal of 
$M^\Uu_{\alpha+1}$,
clearly $\xi+1\in\dropset^\Tt$ iff $\xi+1\in\dropset^\Uu$,
and if $\xi+1\in\dropset^\Tt$ then $M^{*\Tt}_{\xi+1}=M^{*\Uu}_{\xi+1}$.
So in the dropping case, it is easy to maintain the hypotheses
(and $\xi+1$ is non-$\Tt$-special).

Suppose $\xi+1\notin\dropset^\Tt$.
Then $M^{*\Tt}_{\xi+1}=M^\Tt_{\alpha+1}=M^\Tt_\alpha$
and since $\nu\leq\kappa$, we get $\udeg^\Tt(\alpha+1)=0$
and
\[ M^\Tt_{\xi+1}=\Ult_{\udash 0}(M^\Tt_{\alpha+1},E)=\Ult(M^\Tt_{\alpha+1},E), 
\]
so $\xi+1$ is $\Tt$-special.
Noting $\delta<\rho_{e}(M^\Uu_{\alpha+1})$,
$\Uu$ does not drop in model or degree at $\xi+1$,
and $M^\Uu_{\xi+1}=\Ult_e(M^\Uu_{\alpha+1},E)$.

Now $F(M^\Tt_{\xi+1})$ is equivalent to the two-step
iteration
$(F(M^\Tt_{\alpha}), E)$. With $\Tt'$ from above, let 
$\lh(\Tt')=\alpha+\ell+3$
(so $\ell<\om$; note $\lh(\Tt')\geq\alpha+3$
as $\Tt'$ pads at $\alpha$ and 
$E^{\Tt'}_{\alpha+1}=F(M^\Tt_{\alpha+1})=F(M^\Tt_\alpha)$).

Let $\Tt''=\unrvl(\Tt\rest(\xi+2))$
and
\[ \Uu''=\conv(\Tt'')=\Uu\rest(\xi+2)\conc(\emptyset,\ldots,\emptyset).\]
Then an easy induction gives that for each $i\leq\ell$,
\[ M^{\Tt''}_{\xi+1+i}=\Ult_{\udash 0}(M^{\Tt'}_{\alpha+1+i},E) \]
and $F(M^{\Tt''}_{\xi+1+i})$ is equivalent to the two-step
iteration $(F(M^{\Tt'}_{\alpha+1+i}),E)$,
and $\lh(\Tt'')=\xi+\ell+3$,
and recalling $d=\udeg^{\Tt'}(\alpha+2+\ell)$,
note $d=\udeg^{\Tt''}(\xi+2+\ell)$, and
(letting) $N^*=M^{*\Tt'}_{\alpha+2+\ell}=M^{*\Tt''}_{\xi+2+\ell}$,
we have
\[\begin{array}{rclcl} M^{\Tt''}_{\xi+2+\ell}&=&\Ult_{\udash d}(N^*,
F(M^{\Tt''}_{\xi+1+\ell}))\\
&=&\Ult_{\udash d}(\Ult_{\udash d}(N^*,
F(M^{\Tt'}_{\alpha+1+\ell})),E)\\
&=&\Ult_{\udash d}(M^{\Tt'}_{\alpha+\ell+2},E),\end{array}\]
and since $(M^{\Tt'}_{\alpha+\ell+2})^\pm=M^\Uu_{\alpha+1}$
and $d,e=m^{M^{\Tt'}_{\alpha+\ell+2}}(d)$
correspond appropriately
and the  ultrapower maps of $\Tt',\Uu'$ agree appropriately,
and $M^\Uu_{\xi+1}=\Ult_e(M^\Uu_{\alpha+1},E)$, we get
\[ (M^{\Tt''}_{\xi+2+\ell})^\pm=(\Ult_{\udash 
d}(M^{\Tt'}_{\alpha+\ell+2},E))^\pm=M^\Uu_{\xi+1}=M^{\Uu''}_{\xi+1}
=M^{\Uu''}_{\xi+2+\ell}, \]
$e=m^{M^{\Tt''}_{\xi+2+\ell}}(d)$
and the ultrapower maps of $\Tt'',\Uu''$ agree 
 appropriately also.

Regarding part \ref{item:comparability} for $\Tt''$ and
for $X=\{\beta\bigm|\beta\leq^{\ext,\Tt}\xi+1\}$,
we have $X=\{\xi,\xi+1\}$,
and $\widetilde{\xi}=\xi$ and $\widetilde{\xi+1}=\alpha$,
and $\alpha+1\leq^{\Uu''}\xi+1\leq^{\Uu''}\xi+2$, and $\Uu$ does not drop in 
model or degree at 
$\xi+1$ or $\xi+2$ (note that $\Uu''$ pads at $\xi+1$,
so $\pred^{\Uu''}(\xi+2)=\xi+1$ etc). Parts \ref{item:iterated_comparability}
and \ref{item:real_exts_of_Uu} now follow from the above
considerations and by induction applied to $\Tt'$.

\end{sscasethree}
\begin{sscasethree}\label{sscase:chi_T-special} $\chi$ is $\Tt$-special.
 
So $\chi$ is not the successor of a transition point.
It is straightforward to see
\[\nu(F(M^\Tt_\chi))=\sup_{\alpha<\chi}\nu(E^\Tt_\alpha)\leq\sup_{\alpha<\chi}
\widetilde{\nu}^\Tt_\alpha\leq\kappa.\]
(Note that we can have, for example, $\chi=\alpha+1$
and $\nu(E^\Tt_\alpha)<\iota(E^\Tt_\alpha)=\widetilde{\nu}^\Tt_\alpha$,
and in that case, $\nu(F(M^\Tt_\chi))=\nu(E^\Tt_\alpha)$.)
Note that $F(\Ult_{\udash 0}(M^\Tt_\chi,E))$ is equivalent
to the two-step iteration $(F(M^\Tt_\chi),E)$. So
things are almost the same as in Subsubcase 
\ref{sscase:chi_non_T-special_is_trans_succ},
so we leave the details to the reader.

\end{sscasethree}

\end{scasethree}

\begin{scasethree}
 $\xi=\varepsilon+1$ for a transition point $\varepsilon$.
 
 So with $\bar{\Tt}$ as before, and $\mu+\ell+1=\lh(\bar{\Tt})=\lh(\bar{\Uu})$,
we have $\bar{\Tt}\ins\Tt$ and $\bar{\Uu}\ins\Uu$
and $\mu+\ell=\varepsilon$ and $E^\Tt_\varepsilon=\emptyset$
and $E^\Uu_\varepsilon=F(M^\Uu_{\varepsilon})$
and $\OR(M^\Uu_{\varepsilon})<\lh(E^\Tt_{\varepsilon+1})$.
By observations in Subsubcase \ref{sscase:chi_non_T-special_is_trans_succ}, 
either 
\begin{enumerate}[label=(\roman*)]
 \item\label{item:common_ext} $E^\Tt_{\varepsilon+1}
\in\es(M^\Tt_{\varepsilon+1})\inter
\es(M^\Uu_{\varepsilon+1})$, and we set
$E^\Uu_{\varepsilon+1}=E^\Tt_{\varepsilon+1}$,
or
\item\label{item:delayed_ext} $E^\Tt_{\varepsilon+1}=F(M^\Tt_{\varepsilon+1})$,
and we set $E^\Uu_{\varepsilon+1}=\emptyset$.
\end{enumerate}
In case \ref{item:common_ext}
we now proceed as before  with $\xi=\varepsilon+1$.
(Letting $\chi=\pred^\Tt(\varepsilon+2)=\pred^\Uu(\varepsilon+2)$,
if $\chi=\varepsilon+1$,
 it is like Subsubcase \ref{sscase:chi_non_T-special_is_trans_succ};
 and if $\chi<\varepsilon+1$ and we define $\Tt'',\Uu''$
 much as before, then
 $\varepsilon$ is not of form $\widetilde{\beta}$
 (computed with respect to $\Tt''$) for any $\beta+1<\lh(\Tt'')$
 with $E^{\Tt''}_\beta\neq\emptyset$,
 but note that $\varepsilon+1\notin b^{\Uu''}$ in this case.)
In case
 \ref{item:delayed_ext} it is similar, but the role of the pair 
$(E^\Tt_\xi,E^\Uu_\xi)=(E,E)$
 in the previous cases is replaced by the pair 
$(E^\Tt_{\varepsilon+1},E^\Uu_{\varepsilon})$,
which works since these two extenders are equivalent to one another,
and here we have $\widetilde{\varepsilon+1}=\varepsilon$.
\end{scasethree}

\end{casethree}

\begin{casethree}
 $\lh(\Tt)=\lambda+n+1$ where $\lambda$ is a limit
 and $E^\Tt_{\lambda+i}$ is $\Tt$-vs for all $i<n$.
 
 Let $b=[0,\lambda)_\Tt$. Note that $\lambda$ is $\Tt$-special
iff $\alpha$ is $\Tt$-special
for all sufficiently large $\alpha\in b$.
By parts \ref{item:comparability}
and
\ref{item:iterated_comparability}
 for trees of length ${<\lambda}$, $b$ induces
 a $\Uu\rest\lambda$-cofinal branch,
 which has the properties required by
 Definition \ref{dfn:conversion}(\ref{item:conversion_limits}).
 (If $\lambda$ is $\Tt$-special
 then apply part \ref{item:comparability}
 to  $\unrvl(\Tt\rest(\alpha'+1))$ for sufficiently large 
$\alpha'+1<^\Tt\lambda$.)

So if $\lambda$ is non-$\Tt$-special,
then induction easily shows that
$(M^\Tt_\lambda)^\pm=M^\Uu_\lambda$ and iteration maps agree appropriately
etc.
If $\lambda$ is $\Tt$-special, then proceed essentially
as in Subsubcase \ref{sscase:chi_T-special},
but using  $\vec{F}^\Tt_{\alpha\lambda}$
and the equivalent $\vec{F}^\Uu_{\alpha\lambda}$,
where $\alpha\in b$ is sufficiently large,
in place of single extenders of $\Tt,\Uu$.
\end{casethree}

This completes the proof of the lemma.
\end{proof}

\begin{lem}\label{lem:conversion_covers}
 Let $(M,k,m)$ be suitable and $\Uu'$ be an $(M,\uu)$-wellfounded
 $m$-maximal tree on $M$. Then there is a unique pair
 $(\Tt,\Uu)$ such that $\Tt$ is an
 unravelled everywhere unravelable tree $\Tt$
 with $(M,k,\Tt,m)$ suitable, $\Uu=\conv(\Tt)$,
 and $\Uu'$ is given by removing all padding from $\Uu$.
\end{lem}
\begin{proof}[Proof sketch]
The proof is very much like that of Lemma \ref{lem:conversions_cover_J},
and anyway is straightforward.
So we just give a sketch, and the reader should
refer to \ref{lem:conversions_cover_J} for more detail.

We ignore 
$\Uu'$ itself and just directly
discuss $\Uu$.
We proceed by induction on 
$\lh(\Uu)$.
The induction is an easy consequence of Lemma \ref{lem:conversion_exists}
except for the case that $\lh(\Uu)=\lambda+1$ with limit $\lambda$, so consider 
this
assuming that the lemma holds for trees of length ${\leq\lambda}$.
In particular, we have corresponding trees $\Tt\rest\lambda$
and $\Uu\rest\lambda$.

\begin{clm*}
 There is $\alpha<^\Uu\lambda$
 such that for all transition points $\xi$ of $\Tt\rest\lambda$ 
 with $\xi+1\in(\alpha,\lambda)_\Uu$,
 letting $\delta=\lgcd(M^\Tt_\xi)$,
  there is $\chi\in[\xi+1,\lambda)_\Uu$
 such that $i^\Uu_{\xi+1,\chi}(\delta)\leq\crit(i^\Uu_{\chi\lambda})$.
\end{clm*}
\begin{proof}
If not, then select a sequence $\left<(\xi_n,\delta_n)\right>_{n<\om}$
of witnessing pairs $(\xi,\delta)$ with $\xi_n<\xi_{n+1}$. Then just note
that 
since 
$\delta_n\leq i^{*\Uu}_{\xi_n+1}(\crit(E^\Uu_{\xi_n}))$,
we get 
\[ i^\Uu_{\xi_n+1,\lambda}(\delta_n)>i^\Uu_{\xi_{n+1}+1,\lambda}(\delta_{
n+1})\]
for each $n<\om$, so $M^\Uu_\lambda$ is illfounded, a contradiction.
\end{proof}

Now the more complex case is when
there are cofinally many $\eta<^\Uu\lambda$
which are transition points, so consider this case.
Fix $\xi_0<^\Uu\lambda$ with 
$(\xi_0,\lambda)_\Uu\inter\dropset^\Uu_{\deg}=\emptyset$.
For transition points $\eta$ with $\eta+1\in(\xi_0,\lambda)_\Uu$,
let $\chi_\eta$ be the least $\chi$ such that
$i^\Uu_{\eta+1,\chi}(\delta)\leq\crit(i^\Uu_{\chi\lambda})$
and $E^\Uu_\gamma\neq\emptyset$,
where $\gamma+1=\successor^\Uu(\chi,\lambda)$.
For non-transition points $\eta$ such that $\eta+1\in(\xi_0,\lambda)_\Uu$ and 
$E^\Uu_\eta\neq\emptyset$,
let $\chi_\eta=\eta+1$.

Let $\chi$ be least such that either $\chi=\lambda$
or $\xi_0$ is $\Tt$-special and letting $\delta=\lgcd(M^\Tt_{\xi_0})$,
we have $i^\Uu_{\xi_0\chi}(\delta)\leq\crit(i^\Uu_{\chi\lambda})$
and $E^\Uu_\gamma\neq\emptyset$
where $\gamma+1=\successor^\Uu(\chi,\lambda)$.
Let
$b'_0=[0,\xi_0]_\Tt\cup b'$ where $b'$ is the ${<\chi}$-closure
of
\[ \chi\inter\{\chi_\eta\bigm|\eta+1\in(\xi_0,\lambda)_\Uu\text{ and 
}E^\Uu_\eta\neq\emptyset\}.\]

One now shows that $b'_0$ is a branch of $\Tt\rest\lambda$
with $b'_0\cut(\xi_0+1)\inter\dropset^\Tt_{\deg}=\emptyset$.

Suppose $\chi<\lambda$, so $\xi_0$ is $\Tt$-special. Then much as before,
there is $\chi'<^\Uu\chi$
such that $\alpha$ is $\Tt$-vs for each $\alpha\in[\chi',\chi)$,
and $b'_0\cup\{\chi'\}$ is a branch of $\Tt\rest\lambda$,
and $(\xi_0,\chi']_\Tt\inter\dropset_\Tt=\emptyset$.
So letting $\ell_\alpha$ be the $\ell<\om$ 
such that $\unrvl(\Tt\rest(\alpha+1))$
has length $\alpha+1+\ell$, note that $\ell_{\xi_0}=\ell_{\chi'}>\ell_\chi$.
Now set $\xi_1=\chi$. Given $\xi_n$, define $b_n'$
from $\xi_n$ like $b_0'$ was defined from $\xi_0$.
Then we reach some $n<\om$
with $b'_n$ cofinal in $\lambda$.
Define $\Tt'=\Tt\rest\lambda\conc b_n'$.

One can now show that $\Tt'$ has wellfounded models,
$\unrvl(\Tt')$ exists and
\[ \conv(\unrvl(\Tt'))=\Uu\conc(\emptyset,\ldots,\emptyset), \]
which gives what we need.

This completes the sketch of the proof of the lemma.
\end{proof}

\begin{dfn}
 Let $(M,k,m)$ be suitable.
 Let $\Tt$ be a $\udash k$-maximal tree on $M$.
 Note there is a unique nicely padded tree $\Tt'$
 on $M$ which is equivalent to $\Tt$; write $\mathrm{pad}(\Tt)=\Tt'$.
 Say that $\Tt$ is \emph{unravelled} or \emph{everywhere unravelable}
 iff $\Tt'$ is.
 For a padded $m$-maximal tree $\Uu$ on $M^\pm$,
 let $\mathrm{unpad}(\Uu)$ be the
 tree given by removing all padding from $\Uu$.
 We extend the $\conv$ function as follows:
For $\Tt$ unravelled everywhere unravelable $\udash k$-maximal
(without padding in $\Tt$),
define
\[ \conv(\Tt)=\mathrm{unpad}(\conv(\mathrm{pad}(\Tt))).\]
(There is no ambiguity here, because if 
$\Tt=\mathrm{pad}(\Tt)$, i.e. $\mathrm{pad}(\Tt)$ contains
no padding, then $\Uu=\conv(\Tt)$ as defined earlier,
then $\Uu$ contains no padding, so $\mathrm{unpad}(\Uu)=\Uu$.)
Note here that the 
padding, 
depadding, and 
conversion 
processes preserve 
tree length modulo 
$+\om$; in fact,
\begin{enumerate}[label=--]
\item $\lh(\Tt)\leq\lh(\mathrm{pad}(\Tt))<\lh(\Tt)+\om$,
\item $\lh(\conv(\mathrm{pad}(\Tt)))=\lh(\mathrm{pad}(\Tt))$,
and 
\item $\lh(\conv(\Tt))\leq\lh(\conv(\mathrm{pad}
(\Tt)))<\lh(\conv(\Tt))+\om=\lh(\Tt)+\om$.\qedhere
\end{enumerate}
\end{dfn}

\begin{rem}\label{rem:trees_via_strategies}
Let $(M,k,m)$ be suitable
and $\eta\in\OR$.
Let  $\Sigma$ be a $(\udash k,\eta+\om)$-iteration strategy
for $M$ and let $\Tt$ be via $\Sigma$, of successor length.
Then note that $\Tt$ is arbitrarily finitely extendible
(that is, every putative $\udash k$-maximal tree $\Ss$
such that $\Tt\ins\Ss$ has wellfounded models).
It follows that $\Tt$ is everywhere unravelable.
Also, if $\Gamma$ is an $(m,\eta+\om)$-iteration
strategy and $\Uu$ is via $\Gamma$, then $\Uu$ is $M$-$\udash$wellfounded.
(If $\Uu$ has length $\alpha+1$ and $M^\Uu_\alpha$
is active type 3, then $(M^{+\Uu}_\alpha)^\passive\ins 
M^{\Uu'}_{\alpha+1}$,
where $\Uu'=\Uu\conc F(M^\Uu_\alpha)$.)

Now suppose that $\cof(\eta)>\om$ and let $\Sigma$ be an $(m,\eta+1)$-iteration 
strategy
 for $M$. Then $\Sigma$ extends trivially to an $(m,\eta+\om)$-iteration 
strategy for $M$.
 Likewise for $(\udash m,\eta+1)$- and $(\udash m,\eta+\om)$-strategies.
 (Any illfoundedness would reflect down into some tree
 via $\Sigma$ of length $<\eta$.)
\end{rem}

\begin{tm}\label{tm:convert_normal_strategies}
 Let $(M,k,m)$ be suitable and $\eta\in\OR$.
 Then $M$ is $(\udash k,\eta+\om)$-iterable
 iff $M^\pm$ is $(m,\eta+\om)$-iterable.
 Moreover, there is a bijection
 \[ \Sigma\mapsto\conv(\Sigma) \]
 from the $(\udash k,\eta+\om)$-iteration strategies
 $\Sigma$ for $M$ to the $(m,\eta+\om)$-iteration strategies $\conv(\Sigma)$
 for $M^\pm$ such that for each unravelled (and everywhere
 unravelable; see Remark \ref{rem:trees_via_strategies}) $\udash k$-maximal
 tree $\Tt$ on $M$, we have
 \[ \Tt\text{ is via }\Sigma\iff\conv(\Tt)\text{ 
is via }\conv(\Sigma),\]
and therefore the properties
described in Lemma \ref{lem:conversion_exists}
holds.
\end{tm}

So for example, if 
 $\Tt$ has successor length then so does $\Uu=\conv(\Tt)$,
 and $M^\Tt_\infty=M^\Uu_\infty$,
 and if $[0,\infty]_\Tt$ does not drop, then neither
does $[0,\infty]_\Uu$, and
\[ i^\Tt_{0\infty}\rest(M^\pm)^\sq=i^\Uu_{0\infty}.\]

\begin{rem}
Since the proof above functions at a tree-by-tree level,
it can easily be adapted to natural
 kinds of partial strategies; for example,
 strategies which act on trees based on $M|\delta$
 for some $M$-cardinal $\delta<\rho_0(M^\pm)$,
 or trees which use only extenders $E\in\es_+^N$
 such that $\nu_E$ is an $N$-cardinal, etc,
 as these properties are preserved appropriately
 by the conversion processes. Beyond such preservation,
 we just
 need to know that all the $\udash k$-maximal
 trees are everywhere unravelable, and all the
 $m$-maximal trees are $M$-$\udash$wellfounded.
\end{rem}

\begin{dfn}
Let $(M,k,m)$ be suitable.
The iteration game $\Gg_{\mathrm{opt}}(M^\pm,m,\lambda,\eta)^*$
is defined in \cite{iter_for_stacks} (\emph{opt} stands for 
\emph{optimal}).\footnote{The game
builds a putative $m$-maximal stack. The subscript
\emph{opt} means that player I may not make artificial drops,
and the asterisk (introduced in \cite{coremore}) means
 that if in some round $\gamma<\lambda$,
 a tree $\Tt_\gamma$ is produced of length $\eta$
 (with wellfounded well-defined models) then the entire game
 stops and player II has won.}
 The game $\Gg_{\mathrm{opt}}(M,\udash k,\lambda,\eta)^*$
 is completely analogous, but with $\udash$fine structure.
 The game $\Gg^{\unrvl}_{\mathrm{opt}}(M,\udash k,\lambda,\eta)^*$
 makes the restriction that player I may only end rounds
 with unravelled trees. We define via these games
 the corresponding iteration strategies and iterability notions
 (\emph{optimally-$(m,\lambda,\eta)^*$-iterable},
 \emph{unravelled-optimally-$(\udash k,\lambda,\eta)^*$-iterable},
 etc).
\end{dfn}

\begin{tm}
 Let $(M,k,m)$ be suitable and $\eta,\lambda\in\OR$
 with $\eta$ a limit ordinal. 
 Then:
 \begin{enumerate}
 \item $M$ is unravelled-optimally-$(\udash k,\lambda,\eta)^*$-iterable
 iff $M^\pm$ is optimally-$(m,\lambda,\eta)^*$-iterable, and\footnote{Actually
 in this case, the asterisk makes no difference, as $\eta$ is a limit ordinal.}
 \item if $\cof(\eta)>\om$ then $M$ is unravelled-optimally-$(\udash 
k,\lambda,\eta+1)^*$-iterable
 iff $M^\pm$ is optimally-$(m,\lambda,\eta+1)^*$-iterable.
 \end{enumerate}
 Moreover, there are bijections $\Sigma\mapsto\conv(\Sigma)$
 between the sets classes of iteration strategies,
 whose action in each round (that is, for each normal tree in a stack)
 by the conversion of 
Theorem \ref{tm:convert_normal_strategies}.
\end{tm}
\begin{proof}
The iteration games produce putative stacks
$\left<\Tt_\alpha\right>_{\alpha<\gamma}$
and $\left<\Uu_\alpha\right>_{\alpha<\gamma'}$ respectively.
For  round $0$, producing $\Tt_0,\Uu_0$,
we just use Theorem \ref{tm:convert_normal_strategies} (and the conversion 
process used in its proof).
Given $\vec{\Tt}=\left<\Tt_\alpha\right>_{\alpha<\xi}$
and $\vec{\Uu}=\left<\Uu_\alpha\right>_{\alpha<\xi}$,
with $M^{\vec{\Tt}}_\infty$ and $M^{\vec{\Uu}}_\infty$
well-defined and 
wellfounded, we will have that $(M^{\vec{\Tt}}_\infty,k',m')$
is suitable, where $k'=\udeg^{\vec{\Tt}}_\infty$
and $m'=\deg^{\vec{\Uu}}_\infty$,
and that $(M^{\vec{\Tt}}_\infty)^\pm=M^{\vec{\Uu}}_\infty$.
Thus, for round $\xi$, we can again use Theorem 
\ref{tm:convert_normal_strategies}.
The agreement between models, degrees and iteration maps
given by Theorem \ref{tm:convert_normal_strategies} and Lemma 
\ref{lem:conversion_exists}
ensures that the inductive hypotheses  carry through limit stages.
\end{proof}

\begin{rem}
 Like for normal trees, this also adapts easily to partial strategies
 for stacks.
\end{rem}

\bibliographystyle{plain}
\bibliography{stacks_from_normal_iterability_fork_arxiv}

\end{document}